\documentclass[a4paper]{article}
\usepackage{amsmath,amsthm,amssymb}

\usepackage{color}
\usepackage{graphicx} 

\usepackage[all]{xy}

\newtheorem{thm}{Theorem}[section]
\newtheorem{defn}[thm]{Definition}
\newtheorem{ex}[thm]{Example}
\newtheorem{prop}[thm]{Proposition}
\newtheorem{PROP}[thm]{``Proposition''}

\newtheorem{cor}[thm]{Corollary}
\newtheorem{lem}[thm]{Lemma}
\newtheorem{conj}[thm]{Conjecture}
\newtheorem{rem}[thm]{Remark}

\newcommand{\mf}[1]{{\mathfrak{#1}}}
\newcommand{\mr}[1]{{\mathrm{#1}}}
\newcommand{\mb}[1]{{\mathbf{#1}}}
\newcommand{\bb}[1]{{\mathbb{#1}}}
\newcommand{\mca}[1]{{\mathcal{#1}}}

\newcommand{\Hom}{\mr{Hom}}

\newcommand{\Z}{\bb{Z}}

\newcommand{\C}{\bb{C}}

\newcommand{\R}{\bb{R}}

\newcommand{\vv}{\mb{v}}

\usepackage{pst-all}

\numberwithin{equation}{section}

\newcommand{\A}{\mathcal{A}}

\newcommand{\Afd}{\mathcal{A}}

\newcommand{\kk}{\mathbf{k}}
\newcommand{\M}{\mathcal{M}_{\Afd}}
\newcommand{\MM}{\mathcal{M}^{(2)}_{\Afd}}
\newcommand{\sS}{\mathcal{S}}
\newcommand{\tT}{\mathcal{T}}
\newcommand{\uu}{\underline{u}}
\newcommand{\X}{\mathcal{X}}
\newcommand{\Y}{\mathcal{Y}}

\newcommand{\teng}{{}^{\mr{t}}\hspace{-2pt} g}

\newcommand{\da}{\mathcal{D}\Gamma}
\newcommand{\pera}{\mathrm{per}\Gamma}
\newcommand{\dfa}{\mathcal{D}^{\mathrm{fd}}\Gamma}

\newcommand{\moda}{\mathrm{Mod}\hspace{1pt}J}
\newcommand{\modfa}{\mathrm{mod}\hspace{1pt}J}

\newcommand{\too}{\longrightarrow}

\newcommand{\tepk}{\tilde{\varepsilon}_{\mca{C}_\kk}}

\newcommand{\ep}{\varepsilon}
\newcommand{\tep}{\tilde{\varepsilon}}
\newcommand{\hep}{\hat{\varepsilon}}

\newcommand{\MH}{\mr{MH}}
\newcommand{\hMH}{\widehat{\mr{MH}}}
\newcommand{\QT}{\mr{QT}}
\newcommand{\hQT}{\widehat{\mr{QT}}}
\newcommand{\DQT}{\mr{Q}\mathbb{T}}
\newcommand{\hDQT}{\widehat{\DQT}}
\newcommand{\T}{\mr{T}}
\newcommand{\DT}{\mathbb{T}}
\newcommand{\hT}{\widehat{\mr{T}}}
\newcommand{\hDT}{\widehat{\mathbb{T}}}

\newcommand{\pf}{^{\mr{pf}}}

\title{Donaldson-Thomas theory and cluster algebras}
\author{Kentaro Nagao\\
RIMS, Kyoto University\\
  Kyoto 606-8502, Japan
}

\begin{document}

\maketitle
\begin{abstract}
We provide a transformation formula of non-commutative Donaldson-Thomas invariants under a composition of mutations. 
Consequently, we get a description of a composition of cluster transformations in terms of quiver Grassmannians.
\end{abstract}

\setcounter{tocdepth}{2}
\tableofcontents

\section*{Introduction}
Donaldson-Thomas invariants (\cite{thomas-dt, mnop}) are defined as the topological Euler characteristics (more precisely, the weighted Euler characteristics weighted by Behrend function \cite{behrend-dt}) of the moduli spaces of sheaves on a Calabi-Yau $3$-fold (more generally, the moduli spaces of objects in a $3$-Calabi-Yau category \cite{szendroi-ncdt,joyce-4,ks,joyce-song}). 
Dominic Joyce introduced the motivic Hall algebra for an Abelian category in his study of generalized Donaldson-Thomas invariants (\cite{joyce-2}).
One of the important results is that for a $3$-Calabi-Yau category there exists a Poisson algebra homomorphism, so called the {\it integration map}, from the motivic Hall algebra to a power series ring (\cite{joyce-2, joyce-song, bridgeland-hall})
The integration map is given by taking the (weighted) Euler characteristic of an element in the motivic Hall algebra.
Due to the integration map, we get the following powerful method in Donaldson-Thomas theory for $3$-Calabi-Yau categories, which originates with Reineke's computation of the Betti numbers of the spaces of stable quiver representations (\cite{reineke-HN}):
\begin{quote}
Starting from a simple categorical statement, provide an identity in the motivic Hall algebra.
Pushing it out by the integration map, we get a power series identity for the generating functions of Donaldson-Thomas invariants.
\end{quote}
The aim of this paper is to provide 
\begin{itemize}
\item[(1)] Theorem \ref{thm_trans} (Theorem \ref{thm_01}) : a transformation formula of the noncommutative Donaldson-Thomas invariants, and 
\item[(2)] Theorem \ref{thm_CC} (Theorem \ref{thm_02}) and the results in \S \ref{} : its application to the theory of cluster algebras 
\end{itemize}
using this method.

\subsection*{Transformation formula of ncDT invariants}
Let $Q$ be a quiver and $W$ be a potential. 
In this paper, we always assume that 
\begin{itemize}
\item the quiver has the vertex set $I=\{1,\ldots,n\}$, 
\item the quiver has no loops and oriented $2$-cycles, and 
\item the potential is {\it finite}, i.e. a {finite} linear combination of oriented cycles.
\end{itemize}
Let $J=J_{Q,W}$ be the (non-complete) Jacobi algebra.
We have a $3$-Calabi-Yau triangulated category (the derived category of Ginzburg's dg algebra) with a t-structure whose core $\mca{A}$ is the module category of the Jacobi algebra. 
It was proposed by B. Szendroi (\cite{szendroi-ncdt}) to study Donaldson-Thomas theory for the Abelian category $\mca{A}\simeq \mr{mod}J$ ({\it non-commutative Donaldson-Thomas theory}\,).

For a vertex $i\in I$, let $P_i$ denote the projective indecomposable $J$-module corresponding to the vertex $i$. 
For a dimension vector $\mathbf{v}\in (\mb{Z}_{\geq 0})^{I}$, let $\mr{Hilb}_J(i;\mathbf{v})$ be the moduli scheme which parametrizes elements in $V\in \modfa$ equipped with a surjection from $P_i$ such that $[V]=\mathbf{v}$:
\[
\mr{Hilb}_J(i;\mathbf{v}):=\{P_i\twoheadrightarrow V\mid V\in \A, [V]=\mathrm{v} \}.
\]
The (Euler characteristic version of the) non-commutative Donaldson-Thomas invariant is defined by
\[
\mr{DT}_{J,+}(i;\mathbf{v})=e_+({\mr{Hilb}_J(i;\mathbf{v})}):=
e({\mr{Hilb}_J(i;\mathbf{v})})
\]
where $e(\bullet)$ denote the topological Euler characteristic.
In the context of this paper, we will also deal with the invariant 
\[
\mr{DT}_{J,-}(i;\mathbf{v})=e_-({\mr{Hilb}_J(i;\mathbf{v})})
\]
where $e_-(\bullet)$ denote the weighted Euler characteristic weighted by the Behrend function (Definition \ref{defn_ncdt}).

For a vertex $k$, we assume that the mutation $\mu_k(Q,W)$ is well-defined.
Due to the result by Keller and Yang, $(Q,W)$ and $\mu_k(Q,W)$ provide the same derived category with different t-structures (\cite{dong-keller,keller-completion}).
Kontsevich and Soibelman (\cite{ks}) observed that the cluster transformation appears in the transformation formula of non-commutative Donaldson-Thomas invariants under a mutation. 
In this paper, generalizing their observation, we provide a transformation formula of the non-commutative Donaldson-Thomas invariants under a composition of mutations.

We put
\[
\begin{array}{c}
T_{Q,\pm}:=\mathbb{C}\left[(y_{1,\pm})^{\pm 1},\ldots ,(y_{n,\pm})^{\pm 1}\right],\quad 
T^\vee_{Q,\pm}:=\mathbb{C}\left[(x_{1,\pm})^{\pm 1},\ldots ,(x_{n,\pm})^{\pm 1}\right],\vspace{2mm}\\
\mathbb{T}_{Q,\pm}:=T^\vee_{Q,\pm}\otimes_{\mathbb{C}}T_{Q,\pm}
\end{array}
\]
They are called the semiclassical limits of {\it quantum torus}, {\it quantum dual torus} and {\it quantum double torus}\footnote{Since $\mr{Spec}$ of them are algebraic tori, we call them tori with a slight abuse.} respectively.
They are taken as the group algebra of the lattices $M_Q$, $L_Q$ and $M_Q\oplus L_Q$ which are related to the Grothendieck group of the derived category (\S \ref{subsec_Gro}).
Since we have derived equivalences between $(Q,W)$ and $\mu_\mathbf{k}(Q,W)$, we have isomorphisms of the corresponding tori. 
We identify them by these isomorphisms.

We take a certain completion $\widehat{\mathbb{T}}_{Q,\pm}$ of $\mathbb{T}_{Q,\pm}$ (\S \ref{subsub_auto}).
We define the generating function of the Donaldson-Thomas invariants by
\[
\mca{Z}_{J,\pm}^{i}:=
\sum_{\mathbf{v}}\mr{DT}_{J,\pm}(i;\mathbf{v})\cdot y^\mathbf{v}_\pm
\]
where $y^\mathbf{v}_\pm:=\prod (y_{i,\pm})^{v_i}$. 
Using the generating functions, we define algebra automorphisms $\mca{DT}_{J,\pm}$ of $\widehat{\mathbb{T}}_{Q,\pm}$ by
\[
\mca{DT}_{J,\pm}(x_{i,\pm}):=x_{i,\pm}\cdot \mca{Z}_{J,\pm}^{i},\quad 
\mca{DT}_{J,\pm}(y_{i,\pm}):=y_{i,\pm}\cdot \prod_j(\mca{Z}_{J,\pm}^{j})^{\bar{Q}(j,i)}
\]
where 
\[
\bar{Q}(j,i):=Q(i,j)-Q(j,i),\quad Q(i,j)=\sharp \{\text{arrows from $i$ to $j$ in $Q$}\}.
\]
For a sequence of vertices $\mathbf{k}=(k_1,\ldots,k_l)\in I^l$, 
let $\mu_\mathbf{k}(Q,W)$ denote the new QP $\mu_{k_l}(\cdots \mu_{k_1}(Q,W)\cdots)$ and $J_\mathbf{k}$ denote the Jacobi algebra associated to $\mu_\mathbf{k}(Q,W)$.
Then we have two isomorphisms $\mca{DT}_{J,\pm}$ and 
$\mca{DT}_{J_\kk,\pm}$ of the torus \footnote{To be precise, they are isomorphisms of different completions. See Theorem \ref{thm_trans} for the precise statement.}.
Our transformation formula of DT invariants is given as the relation of these isomorphisms. 

In \S \ref{subsec_grass}, we construct a $J$-module $R_{\kk,i}$ and 
define the {\it quiver Grassmannian} 
which parametrizes quotient modules of $R_{\kk,i}$ :
\[
\mathrm{Grass}(\mathbf{k};i,\mathbf{v}):=
\{R_{\kk,i}\twoheadrightarrow V\mid V\in \A,\ [V]=\mathbf{v}\}.
\]
The formula is described in terms of (weighted) Euler characteristics of the quiver Grassmannians.
\begin{thm}\textup{($=$ Theorem \ref{thm_trans}, transformation formula of ncDT invariants)}\label{thm_01}
Assume that the $(Q,W)$ is successively f-mutatable with respect to the sequence $\kk$ (see \S \ref{subsub_pmutation} for the details of the assumption). 
Then we have the following ``commutative diagram'' \footnote{This diagram is not rigorous in that the compositions of the maps are not well-defined. See Theorem \ref{thm_trans} for the precise statement.}:
\[
\xymatrix{
\widehat{\mathbb{T}}_{Q_\kk ,\pm}
\ar[rr]^{\mr{Ad}_{\mca{T}_\kk[-1],\pm}}
\ar[d]_{\mca{DT}_{\hspace{-1mm}J_\kk,\pm}}
& &
\widehat{\mathbb{T}}_{Q,\pm}
\ar[d]^{\mca{DT}_{\hspace{-1mm}J,\pm}}\\
\widehat{\mathbb{T}}_{Q_\kk,\pm}
\ar[rr]_{\mr{Ad}_{\mca{T}_\kk,\pm}}
& &
\widehat{\mathbb{T}}_{Q,\pm}.
}
\]
The morphism $\mr{Ad}_{\mca{T}_\kk[-1],\pm}$ is given by 
\begin{align}
\mr{Ad}_{\mca{T}_\kk[-1],\pm}(x_{\kk,i,\pm})
&=x_{\kk,i,\pm}\cdot \Biggl(\sum_\mathbf{v}e_\pm\Bigr(\mathrm{Grass}(\mathbf{k};i,\mathbf{v})\Bigr)\cdot\mathbf{y}_\pm^{-\mathbf{v}}\Biggr),\\
\mr{Ad}_{\mca{T}_\kk[-1],\pm}(y_{\kk,i,\pm})
&=y_{\kk,i,\pm}\cdot \prod_j\Biggl(\sum_\mathbf{v}e_\pm\Bigr(\mathrm{Grass}(\mathbf{k};j,\mathbf{v})\Bigr)\cdot\mathbf{y}_\pm^{-\mathbf{v}}\Biggr)^{\bar{Q}(j,i)}.
\end{align}
where $x_{\kk,i,\pm}$ and $y_{\kk,i,\pm}$ are generators of ${\mathbb{T}}_{Q_\kk ,\pm}$\footnote{The variables $x_{\kk,i,\pm}$ and $y_{\kk,i,\pm}$ on the left hand side of the equations does make sense since we have identified the two tori ${\mathbb{T}}_{Q_\kk ,\pm}$ and ${\mathbb{T}}_{Q,\pm}$.}.
The morphism $\mr{Ad}_{\mca{T}_\kk,\pm}$ is given by 
\begin{equation}\label{eq_Sigma}
\mr{Ad}_{\mca{T}_\kk,\pm}
:=\Sigma\circ \mr{Ad}_{\mca{T}_\kk[-1],\pm}\circ \Sigma
\end{equation}
where $\Sigma$ is the involution of the tori given by 
\[
\Sigma(x_{i,\pm})=(x_{i,\pm})^{-1},\quad
\Sigma(y_{i,\pm})=(y_{i,\pm})^{-1}.
\]
\end{thm}
If we take a sequence $\kk=(k)$ of length $1$, then we have
\[
R_{(k),i}=\begin{cases}
0 & i\neq k,\\
s_k & i=k.
\end{cases}
\]
Hence we have
\begin{equation}\label{eq_preCT}
\mr{Ad}_{\mca{T}_{(k)}[-1],\pm}(x_{(k),i,\pm})=
\begin{cases}
x_{(k),i,\pm} & i\neq k,\\
x_{(k),k,\pm}(1+(y_{k,\pm})^{-1}) & i=k.
\end{cases}
\end{equation}
This recovers the results in \cite[pp143]{ks}.

\subsection*{Composition of cluster transformations}
Cluster algebras were introduced by Fomin and Zelevinsky (\cite{fomin-zelevinsky1}) in their study of dual canonical bases and total positivity in semi-simple groups. 
Although the initial aim has not been established, it has been discovered that the theory of cluster algebras has many links with a wide range of mathematics (see \cite[\S 1.1]{keller-survay} and the references there).
Since a cluster transformation helps us to understand the whole structure in an inductive way, study of compositions of cluster transformations is important.

A {\it seed} is a pair $(Q\mid\underline{u})$, where
\begin{enumerate}
\item $Q$ is a quiver without loops and oriented $2$-cycles, and 
\item $\uu=(u_1,\ldots,u_n)$ is a free generating set of the field $\C(x_1,\ldots,x_n)$.
\end{enumerate}
For a vertex $k\in I$, the {\it mutation} $\mu_k(Q\mid\uu)$ of $(Q\mid\uu)$ at $k$ is the seed $(\mu_kQ\mid\uu^{\mr{new}})$, where $\mu_kQ$ is the mutation of the quiver (\S \ref{qmutation}) and $\uu^{\mr{new}}$ is obtained from $\uu$ by replacing $u_k$ with
\begin{equation}\label{eq_cluster}
u_k^{\mr{new}}=u_k^{-1}\Biggl(\prod_i (u_k)^{Q(i,k)}+\prod_i (u_k)^{Q(k,i)}\Biggr)
\end{equation}
This is called the {\it cluster transformation}.
Given a quiver $Q$, we call $(Q\mid\underline{x})=(Q,(x_1,\ldots,x_n))$ an {\it initial seed}.
\begin{defn}
For a sequence of vertices $\mathbf{k}=(k_1,\ldots,k_l)\in I^l$ and a vertex $i\in I$, we define rational functions $FZ_{\mathbf{k},i}(\underline{x})$ by 
\[
\mu_{k_l}(\cdots(\mu_2(\mu_1(Q\mid\underline{x}))\cdots)
=(Q_\kk\mid(FZ_{\mathbf{k},i}(\underline{x}))).
\]
\end{defn}
In the case of a quiver of finite type, Caldero and Chapoton (\cite{caldero-chapoton}) described a composition of cluster transformations in terms of quiver Grassmannians of the original quiver.
This result is generalized by many people (see the references in \cite{plamondon} for example).
Finally, Derksen-Weyman-Zelevinsky and Plamondon (\cite{DWZ2,plamondon}) provided the Caldero-Chapoton type formula for an arbitrary quiver without loops and oriented $2$-cycles.
In this paper, we provide an alternative proof of the Caldero-Chapoton type formula 
under the assumption that there is a potential $W$ such that the QP $(Q,W)$ is successively f-mutatable with respect to the sequence $\kk$ (\S \ref{subsub_pmutation}).

We identify $\C(x_1,\ldots,x_n)$ with the fractional field of $T_{Q,+}$. 
We will omit ``$+$'' in the notations.
\begin{thm}\textup{(Caldero-Chapoton type formula)} \label{thm_02}
We have
\begin{equation}\label{eq_CC}
\mathrm{FZ}_{\mathbf{k},i}(\underline{x})=
x_{\kk,i}\cdot \Biggl(\sum_\mathbf{v}e\Bigr(\mathrm{Grass}(\mathbf{k};i,\mathbf{v})\Bigr)\cdot\mathbf{y}^{-\mathbf{v}}\Biggr).
\end{equation}
where $(\underline{y})^{-\mathbf{v}}=\prod_j (y_j)^{-v_j}$ and $y_j=\prod_i(x_i)^{\bar{Q}(i,j)}$.
\end{thm}

\subsection*{Application to cluster algebras}
In \cite{DWZ2,plamondon}, they prove six conjectures given in \cite{fomin-zelevinsky4} for cluster algebras associated to quivers \footnote{Cluster algebras are associated not only with quivers without loops and oriented 2-cycles (equivalently, with skew-symmetric integer matrices)  but also with {\it skew-symmetrizable} matrices,}. 
In \S \ref{subsec_F_and_g} and \S \ref{subsec_g_to_F} we give alternative proofs for them under the assumption that the quiver with principal framing is successively f-mutatable.
\footnote{From the view points of applications to cluster algebras, the finite assumption is too strong. In this sense, our result on the Fomin-Zelevinsky conjectures is weaker than ones in \cite{DWZ2,plamondon}.}.

Let $Q\pf$ be the following quiver:
\begin{description}
\item{vertices} : $I\sqcup I^*$ where $I^*=\{1^*,\ldots,n^*\}$,
\item{arrows} : $\{\text{arrows in $Q$}\}\sqcup \{i^*\to i\mid i\in I\}$.
\end{description}
This is called the quiver with the principal framing associated to $Q$. 
Let us use $\{X_i\}$ and $\{Y_i\}$ for generators of the tori associated to $Q\pf$.
\begin{defn}
\begin{itemize}
\item [(1)]The $F$-polynomial associated to $(Q,W)$, $\kk$ and $i$ is the following :
\[
F_{\kk,i}(\underline{y}):={FZ}\pf_{\kk,i}(\underline{X})|_{X_i=1,X_{i^*}=y_i}.
\]
\item[(2)] 
The $g$-vector $g_{\kk,i}\in M_Q$ associated to $(Q,W)$, $\kk$ and $i$ is the element which is characterized by the following identity :
\[
\mathrm{FZ}_{\mathbf{k},i}(\underline{x})=
\mathbf{x}^{g_{\kk,i}}
\cdot\mr{F}_{\kk,i}(\underline{y}^{-1})
\]
where the last term is given by substituting $y_i^{-1}$ to $y_i$.
\end{itemize}
\end{defn}
\begin{rem}
It is $y_i^{-1}$ in our notation what is denoted by $y_i$ in Fomin-Zelevinsky's notation. 
We use this notation since $y_i$ corresponds to the simple module in our notation.
\end{rem}

The potential $W$ of $Q$ can be taken as a potential of $Q\pf$. 
We {\it assume} that $(Q\pf,W)$ is successively f-mutatable with respect to the sequence $\kk$.

We will apply an argument similar to the one in \S \ref{sec_proof}, for $(Q\pf,W)$.
Then we get descriptions of $g$-vectors and $F$-polynomials in terms of the $3$-Calabi-Yau category : 


\vspace{3mm}

{\renewcommand\arraystretch{1.5}
\begin{tabular}{c|p{0.5\textwidth}}
cluster algebra & \hspace{20mm}DT theory\\\hline\hline
$y$-variable $y_i$ & formal variable corresponding to the simple module $s_i$\\\hline
$x$-variable $x_i$ & formal variables corresponding to the projective module $P_i$ (or $\Gamma_i$)\\ \hline
$F$-polynomial & generating function of the Euler characteristics of the quiver Grassmannians \\\hline
$g$-vector & $\phi_{\kk}^{-1}([\Gamma_{\kk,i}])\in M_Q=K_0(\mr{per}\Gamma)\simeq \Z^I$
\\\hline
${}^tg$-vector & $\phi_{\kk}([s_i])\in L_{Q_\kk}\simeq \Z^I$
\\\hline
$c$-vector & $\phi_{\kk}^{-1}([s_{\kk,i}])\in L_{Q}\simeq \Z^I$
\\\hline
sign coherence of ${}^tg$-vectors & $s_i\in \mca{T}_\kk\subset \mca{A}_\kk[1]$ or $s_i\in \mca{F}_\kk\subset \mca{A}_\kk$\\\hline
sign coherence of $c$-vectors & $s_{\kk,i}\in \mca{T}_\kk[-1]\subset \mca{A}[-1]$ or $s_{\kk,i}\in \mca{F}_\kk\subset \mca{A}$\\\hline
$g$-vectors determine $F$-vectors & 
Bridgeland stability on walls
\end{tabular}}

\vspace{3mm}

\subsection*{Contents}
From \S \ref{sec_1} to \S \ref{subsec_stab}, we study some categorical properties of the $3$-dimensional Calabi-Yau category associated to a quiver with a potential. 
The statements of our main results appear in \S \ref{sec_statement}. 

We prove the theorems using motivic Hall algebra, on which we give a brief review in \S \ref{sec_MH}.
For the proof, first, we show in \S \ref{subsec_eq_in_MH} some identities on the motivic Hall algebra using the results from \S \ref{sec_1} to \S \ref{subsec_stab}.
They are translated in \S \ref{subsec_auto} into the main results via the integration map.

Finally, we study quivers with principal framings to provide alternative proofs for the six conjectures given in \cite{fomin-zelevinsky4} (\S \ref{sec_F_and_g}).

\subsection*{Comments}
\noindent (1) 
Throughout this paper, we assume that all the potentials are finite.
As we mentioned, from the view points of applications to cluster algebras, we would like to remove the assumption.
If we take an infinite potential, then the moduli spaces will not be schemes (or stacks) but formal schemes (or stacks).
Once we construct a theory of the motivic Hall algebra 
in the formal setting, we can apply all the arguments in this paper.

\smallskip

\noindent (2) 
A typical example of a finite potential is a potential associated to a triangulated surface \cite{QP-surface}.
We will apply the results in this paper for a triangulated surface in \cite{DT_surface}.

\smallskip

\noindent (3)
It is expected that there is a refinement of the DT theory, which is called the {\it motivic DT theory} (\cite{ks,behrend_bryan_szendroi}). 
Wall-crossing phenomena of the motivic DT theory has been studied in \cite{ks,WC_motivic}.
We hope to study quantum cluster algebras from the view point of motivic DT theory in the future.

\subsection*{Acknowledgement}
I would like to express my gratitude for all of the following mathematicians; 
Bernhard Keller who patiently explained many things about the cluster categories and the cluster algebras, indicated many stupid mistakes in the very preliminary version of this paper;
Tom Bridgeland who showed me the preliminary version of his paper \cite{bridgeland-hall} and gave me a lot of helpful comments and encouragement. In particular, the proof of Theorem \ref{thm_tilting} is due to him; 
Pierre-Guy Plamondon who kindly explained the results in his PhD thesis \cite{plamondon};
Hiraku Nakajima who explained me his results in \cite{nakajima-cluster} and encouraged me to promote the result of \cite{ks};
Bernard Leclerc who recommended me to give alternative proofs for the conjectures in \cite{fomin-zelevinsky4};
Andrei Zelevinsky who gave me some comments on the preliminary version of this paper.

The first version of this paper was written while I have been visiting the University of Oxford.
I am grateful to Dominic Joyce for the invitation and to the Mathematical Institute for hospitality. 

The author is supported by the Grant-in-Aid for Research Activity Start-up (No. 22840023) and for Scientific Research (S) (No. 22224001).

\section{Preliminary}\label{sec_0}
\subsection{QP, dga and Jacobi algebra}
A quiver with a potential (QP, in short) is a pair $(Q,W)$ of a quiver $Q$ and a potential $W$, a linear combination of oriented cycles.
We say that $W$ (or $(Q,W)$) is {\it finite} when $W$ is a {finite} linear combination of oriented cycles.  
In this paper, we always assume that a QP is finite.

First, we define the derivation of the potential.
For an arrow $a$ and a oriented cycle $a_1\cdots a_l$, we put
\[
\partial_a(a_1\cdots a_l)
:=
\sum_i\delta_{a,a_i}a_{i+1}\cdots a_la_1\cdots a_{i-1}.
\]
For an arrow $a$ and a potential $W$, we define the derivation $\partial_aW$ by the linear combination of the derivations of the oriented cycles.

For a QP $(Q,W)$, we define Ginzburg's differential graded algebra $\Gamma=\Gamma_{Q,W}$.
As a graded algebra, $\Gamma_{Q,W}$ is given by the path algebra $\mathbb{C}\hat{Q}$ of the following 
graded quiver $\hat{Q}$. The vertex set of $\hat{Q}$ is the same as $Q$ and the arrow set is the union of the following three sets :
\begin{itemize}
\item arrows in $Q$ (degree $0$),
\item opposite arrow $a^*$ for each arrow $a$ in $Q$ (degree $-1$),
\item loop $t_i$ at $i$ for each vertex $i$ in $Q$ (degree $-2$).
\end{itemize}
We define the differential $d=d_W$ of degree $1$ on the path algebra $\mathbb{C}\hat{Q}$ as follows：
\begin{itemize}
\item $da=0$ for any arrow $a$ in $Q$，
\item $d(a^*) = \partial_aW$ for any arrow $a$ in $Q$，and 
\item $d(t_i) = e_i\left(\prod_{a}[a,a^*]\right)e_i$ for any vertex $i$ in $Q$.
\end{itemize}
\begin{defn}
\begin{itemize}
\item[\textup{(1)}] The differential graded algebra $\Gamma_{Q,W}=(\mathbb{C}\hat{Q},d_W)$ is called the {\it Ginzburg differential graded algebra} (dga, in short).
\item[\textup{(2)}] The algebra $J=J_{Q,W}:=H^0\Gamma_{Q,W}$ is called the {\it Jacobi algebra}.
\end{itemize}
\end{defn}
The Jacobi algebra can be described as the quiver with the relations :
\[
J_{Q,W}=\mathbb{C}Q/\langle\partial_a W;a\in Q_1\rangle.
\]

\subsection{Quiver mutation}\label{qmutation}
In this paper, we always assume that a quiver has
\begin{itemize}
\item the vertex set $I=\{1,\ldots,n\}$, and
\item no loops and oriented $2$-cycles.
\end{itemize}
For vertices $i$ and $j\in I$, we put
\[
Q(i,j)=\sharp\{\text{arrows from $i$ to $j$}\},\quad \bar{Q}(i,j)=Q(i,j)-Q(j,i).
\]
Note that the quiver $Q$ is determined by the matrix $\bar{Q}(i,j)$ under the assumption above.

For the vertex $k$, we define the new quiver $\mu_kQ$ as follows :
\begin{itemize}
\item
First, we define a new quiver $\mu_k^\mr{pre}{Q}$ as follows ：
\begin{itemize}
\item 
For any subquiver $u\overset{\alpha}{\to} k\overset{\beta}{\to}v$,
we associate a new arrow $[\beta\alpha]\colon u \to v$.
\item replace any arrow $a$ incident to the vertex $k$ with an opposite arrow $a^*$.
\end{itemize}
\item
Remove all oriented cycles of length $2$ in $\mu_k^\mr{pre}{Q}$.
\end{itemize}

\subsection{QP mutation}
\subsubsection{Reduced part of a potential}
Let $\widehat{\mathbb{C}Q}$ be the completion of $\mathbb{C}Q$ with respect to path lengths.

A potential of $Q$ is an element in $\widehat{\mb{C}Q}$ which is described as a linear combination of oriented cycles in $Q$.
We identify two potentials which are related via rotations of oriented cycles.  
A potential is said to be finite if it is an element in $\mathbb{C}Q$.

Two QP $(Q,W)$ and $(Q',W')$ are said to be right equivalent, which is denoted by $(Q,W)\sim (Q',W')$, if
there exists an algebra isomorphism $\psi$ between $\widehat{\mathbb{C}Q}$ and $\widehat{\mathbb{C}Q'}$ so that $\psi(W)=W'$. 
Two finite QP $(Q,W)$ and $(Q',W')$ are said to be right f-equivalent, which is denoted by $(Q,W)\overset{\mr{fin.}}{\sim} (Q',W')$, if
there exists an algebra isomorphism $\psi$ between ${\mathbb{C}Q}$ and ${\mathbb{C}Q'}$ so that $\psi(W)=W'$.

A potential is said to be reduced if it has no oriented cycles of length less than $3$, 
and said to be trivial if its Jacobi algebra is trivial.
For quivers $Q$ and $Q'$ with the same vertex set, let $Q\cup Q'$ denote the quiver given by taking union the arrow sets.
For QPs $(Q,W)$ and $(Q',W')$ with the same vertex set, we take $W$ and $W'$ as potentials of $Q\cup Q'$ and let $(Q,W)\oplus(Q',W')$ denote the new QP $(Q\cup Q',W+W')$.．

For any QP $(Q,W)$, we have a right equivalence
\[
(Q,W)
\sim 
(Q,W)_\mr{red}
\oplus
(Q,W)_\mr{triv}
\]
with reduced $W_\mr{red}$ and trivial ${W_\mr{triv}}$ (\cite[Lemma 4.6]{quiver-with-potentials})．
Moreover, $(Q_\mr{red},W_\mr{red})$ and $(Q_\mr{triv},W_\mr{triv})$ are determined uniquely up to right equivalences．
We call $(Q_\mr{red},W_\mr{red})$ as the reduced part of $(Q,W)$.

A finite QP $(Q,W)$ is said to be f-reducible if we have a right f-equivalence
\[
(Q,W)
\overset{\mr{fin.}}{\simeq} 
(Q_\mr{red},W_\mr{red})
\oplus
(Q_\mr{triv},W_\mr{triv})
\]
with finite reduced $(Q_\mr{red},W_\mr{red})$ and finite trivial $(Q_\mr{triv},W_\mr{triv})$.

\subsubsection{Potential mutation}\label{subsub_pmutation}
For a QP $(Q,W)$ and a vertex $k$, 
we define the potential $\mu_k^\mr{pre}{W}$ of the quiver $\mu_k^\mr{pre}Q$ by
\[
\mu_k^\mr{pre}{W}:=[W]+\Delta
\]
where
\begin{itemize}
\item
$[W]$ is the potential which is obtained from $W$ by replacing all the composition $u\overset{\alpha}{\to} k\overset{\beta}{\to}v$ with $[\beta\alpha]$, and
\item
$\Delta:=\sum \alpha^*\beta^*[\beta\alpha]$.
\end{itemize}
The mutation $\mu_k(Q,W)$ of the QP $(Q,W)$ at $k$ is the reduced part  
$(\mu_k^\mr{pre}{Q},\mu_k^\mr{pre}{W})_\mr{red}$
of $(\mu_k^\mr{pre}{Q},\mu_k^\mr{pre}{W})$.
\begin{defn}
\begin{itemize}
\item[(1)]
We say that a QP $(Q,W)$ is mutatable at $k$ 
if the underlying quiver of $\mu_k(Q,W)$ is $\mu_kQ$, the mutation of the quiver defined in \S \ref{qmutation}.
\item[(2)]
We say that a finite QP $(Q,W)$ is f-mutatable at $k$ if it is mutatable and $(\mu_k^\mr{pre}{Q},\mu_k^\mr{pre}{W})$ is f-reducible.
\end{itemize}
\end{defn}
Let $\kk=(k_1,\ldots,k_l)$ be a sequence of vertices. 
A finite QP $(Q,W)$ is said to be successively f-mutatable with respect to the sequence $\kk$ 
if 
\[
\mu_{k_{s-1}}(\cdots (\mu_{k_1}(Q,W))\cdots)
\]
is f-mutatable at $k_s$.


\section{Derived categories}\label{sec_1}
\subsection{Categories}\label{subsubsec_111}
For a QP $(Q,W)$, 
we have the following triangulated categories :
\begin{description}
\item{$\da$ : }the derived category of right dg-modules over Ginzburg dga $\Gamma$,
\item{$\pera$ : }the smallest full subcategory of $\da$ containing $\Gamma$ and closed under extensions, shifts and direct summands,
\item{$\dfa$ : }the full subcategory of $\da$ consisting of dg-modules with finite dimensional cohomologies.
\end{description}

The triangulated categories $\da$ and $\dfa$ have the canonical t-structures whose cores are 
\begin{description}
\item{$\moda$ : }the category of finitely generated right modules over the (non-complete) Jacobi algebra, and
\item{$\modfa$ : }the full subcategory of $\moda$ consisting of finite dimensional modules
\end{description}
respectively. 
For a vertex $i\in I$, we have the following objects:
\begin{description}
\item{$s_i$ : }the simple $J$-module,
\item{$\Gamma_i:=e_i\Gamma$ : }the $\Gamma$-module , which is a direct summand of $\Gamma$, and 
\item{$P_i:=H^0_{\moda}(\Gamma_i)$ : }the projective indecomposable $J$-module.
\end{description}
Here $e_i$ is the idempotent.

\subsection{Grothendieck groups}\label{subsec_Gro}
We put $M=M_Q:=K_0(\pera)$ and $L=L_Q:=\mathbb{Z}^{I}$, where $L$ is taken as the target of the map 
\[
K_0(\dfa)\to \mathbb{Z}^{I}=L
\]
defined by $[E]\mapsto \underline{\mathrm{dim}}(E)$.
With a slight abuse of notations, we will write $[E]\in L$ instead of $\underline{\mathrm{dim}}(E)$.

We put $M_{\R}=M_{Q,\R}:=M_Q\otimes \R$ and $L_\R=L_{Q,\R}:=L_Q\otimes \R$.
Let $\chi$ denote the Euler pairing $L\times L\to \Z$ given by
\[
\chi([E],[F])=\sum_i(-1)^i\dim\mathrm{Hom}(E,F[i]).
\]
We put $\mathbf{w}_i:=[\Gamma_i]$ and $\mathbf{v}_i:=[s_i]$.
The set $\{\mathbf{w}_i\}$ forms a basis of $M$ and the set $\{\mathbf{v}_i\}$ forms a basis of $L$. 
We extend $\chi$ on $L\otimes M$ by
\[
\chi(\mathbf{w}_i,\mathbf{v}_j):=\delta_{i,j},\quad \chi(\mathbf{w},\mathbf{w'})=0
\]
for any $\mathbf{w}, \mathbf{w}'\in M$. 
This gives $M_{Q,\R}\simeq (L_{Q,\R})^*$.

\subsection{Tori}\label{subsub_torus} 
Let $\sigma$ be a sign; $\sigma=\pm$.
We define $\mathrm{T}^\vee_{Q,\sigma}$, $\mathrm{T}_{Q,\sigma}$ and $\mathbb{T}_{Q,\sigma}$ by 
\[
\begin{array}{c}
\mathrm{T}_{Q,\sigma}^\vee:=\bigoplus_{\mathbf{w}\in M}\C\cdot \mathbf{x}^{\mathbf{w}}_{\sigma},\quad
\mathrm{T}_{Q,\sigma}:=\bigoplus_{\mathbf{v}\in L}\C\cdot \mathbf{y}^{\mathbf{v}}_{\sigma},\vspace{5pt}\\
\mathbb{T}_{Q,\sigma}:=\mathrm{T}_{Q,\sigma}^\vee\otimes\mathrm{T}_{Q,\sigma},
\end{array}
\]
with the following products:
\[
\mathbf{x}^{\mathbf{w}}_{\sigma}\cdot \mathbf{x}_{\sigma}^{\mathbf{w'}}=\mathbf{x}^{\mathbf{w+w'}}_{\sigma}, \quad
\mathbf{y}^{\mathbf{v}}_{\sigma}\cdot \mathbf{y}^{\mathbf{v'}}_{\sigma}=\sigma^{\chi(\mathbf{v},\mathbf{v'})}\mathbf{y}_{\sigma}^{\mathbf{v+v'}}, \quad 
\mathbf{x}^{\mathbf{w}}_{\sigma} \cdot \mathbf{y}^{\mathbf{v}}_{\sigma} =
\mathbf{y}^{\mathbf{v}}_{\sigma} \cdot \mathbf{x}^{\mathbf{w}}_{\sigma}
\]
where we identify $\sigma$ with $\pm 1$.
We put $x_{i,\sigma}:=\mathbf{x}^{[\Gamma_i]}_\sigma$ and $y_{i,\sigma}:=\mathbf{y}^{[s_i]}_\sigma$, then we have
\[
\begin{array}{c}
\mathrm{T}_{Q,\sigma}^\vee=\C[x_{1,\sigma}^\pm,\ldots,x_{n,\sigma}^\pm],\quad
\mathrm{T}_{Q,\sigma}=\C[y_{1,\sigma}^\pm,\ldots,y_{n,\sigma}^\pm],\vspace{5pt}\\
\mathbb{T}_{Q,\sigma}=\C[x_{1,\sigma}^\pm,\ldots,x_{n,\sigma}^\pm,y_{1,\sigma}^\pm,\ldots,y_{n,\sigma}^\pm].
\end{array}
\]
They are called the semiclassical limits of {\it quantum torus}, {\it quantum dual torus} and {\it quantum double torus} respectively \footnote{Since $\mr{Spec}$ of them are algebraic tori, we call them tori with a slight abuse.}.
We define the surjective algebra homomorphism $\pi_\sigma\colon \mathbb{T}_{Q,\sigma}\twoheadrightarrow \mathrm{T}_{Q,\sigma}^\vee$ by
\begin{equation}\label{eq_pi}
x_{i,\sigma} \otimes 1  \longmapsto  x_{i,\sigma},\quad 1 \otimes y_{i,\sigma}  \longmapsto  \mathbf{x}_\sigma^{[s_i]}.
\end{equation}
The kernel of $\pi_\sigma$ is generated by $\{(\mathbf{x}_\sigma^{[s_i]}\otimes 1) - (1\otimes y_{i,\sigma})\mid i \in I\}$.
We sometimes identify an element in $\T_{Q,\sigma}$ with its image in $\T_{Q,\sigma}^\vee$ under the composition
\[
\T_{Q,\sigma} \hookrightarrow \mathbb{T}_{Q,\sigma} \overset{\pi_\sigma}{\twoheadrightarrow}\T_{Q,\sigma}^\vee.
\]
Let $\Sigma$ denote the automorphism of the tori given by 
\[
\Sigma(\mathbf{x}_\sigma^\mathbf{w})=\mathbf{x}_\sigma^{-\mathbf{w}},\quad
\Sigma(\mathbf{y}_\sigma^\mathbf{v})=\mathbf{y}_\sigma^{-\mathbf{v}}.
\]

\subsection{Mutation and derived equivalence}
\subsubsection{Derived equivalence}\label{subsub_derived_equiv}
Let $(Q,W)$ be a finite QP which is f-mutatable at a vertex $k$.
Let $\mu_k\Gamma$ be the Ginzburg dga associated to the mutation $\mu_k(Q,W)$.
\begin{thm}[\protect{\cite[Theorem 3.2]{dong-keller}, \cite[\S 7.6]{keller-completion}}]\label{thm_KY}
There exist equivalences of triangulated categories 
\[
\Phi_{k,+},\Phi_{k,-}\ \colon \da\overset{\sim}{\longrightarrow} \mathcal{D}(\mu_k\Gamma)\]
such that 
\begin{itemize}
\item $\Phi_{k,\pm}^{-1}(\Gamma_i')=\Gamma_i$ for $i\neq k$, and
\item 
$\Phi_{k,+}^{-1}(\Gamma_k')$ and  $\Phi_{k,-}^{-1}(\Gamma_k')$ are involved in the following triangles :
\begin{equation*}
\begin{array}{ccccccc}
\Phi_{k,+}^{-1}(\Gamma'_k)[-1]&\to& \bigoplus_{j}\Gamma_{j}^{\oplus Q(k,j)}&\to& \Gamma_k&\to &\Phi_{k,+}^{-1}(\Gamma'_k),\vspace{1.5mm}
\\
\Phi_{k,-}^{-1}(\Gamma'_k)&\to& \Gamma_k&\to & \bigoplus_{j}\Gamma_{j}^{\oplus Q(j,k)} &\to& \Phi_{k,-}^{-1}(\Gamma'_k)[1]
\end{array}\end{equation*}
\end{itemize}
where $\Gamma_j'$ is the direct summand of $\mu_k\Gamma$.
Moreover, $\Phi_{k,\pm}$ restricts to equivalences from $\pera$ to $\mathrm{per}(\mu_k\Gamma)$ and from $\dfa$ to $\mathcal{D}^{\mathrm{fd}}(\mu_k\Gamma)$.
\end{thm}
\begin{rem}
It is $\Phi_{k,-}^{-1}$ that is studied in \cite[Theorem 3.2]{dong-keller}.
\end{rem}

The equivalences induce isomorphisms
\[
\phi_{k,\pm}\colon M_Q \overset{\sim}{\longrightarrow} M_{\mu_kQ}
\]
and
\[
\phi_{k,\pm}\colon L_Q \overset{\sim}{\longrightarrow} L_{\mu_kQ}.
\]
By the triangles in Theorem \ref{thm_KY} we have 
\begin{equation}\label{eq_yy'}
\begin{array}{l}
\phi_{k,+}^{-1}([\Gamma_i'])=\begin{cases}
[\Gamma_i] & i\neq k, \\
-[\Gamma_k]+\sum_j Q(k,j)[\Gamma_j] & i=k,
\end{cases}\\
\phi_{k,-}^{-1}([\Gamma_i'])=\begin{cases}
[\Gamma_i] & i\neq k, \\
-[\Gamma_k]+\sum_j Q(j,k)[\Gamma_j] & i=k
\end{cases}
\end{array}
\end{equation}
in $M_Q$. 
Since $\phi_{k,\pm}$ preserves $\chi$ we have
\begin{equation}\label{eq_phi_for_s}
\begin{array}{l}
\phi_{k,+}^{-1}([s_i'])=
\begin{cases}
[s_i]+ Q(k,i)[s_k] & i\neq k, \\
-[s_k] & i=k,
\end{cases}\\
\phi_{k,-}^{-1}([s_i'])=\begin{cases}
[s_i]+Q(i,k)[s_k] & i\neq k, \\
-[s_k] & i=k
\end{cases}
\end{array}
\end{equation}
in $L_Q$.
Note that $\phi_{k,\pm}$ also induce isomorphisms between $\mathrm{T}_{Q,\sigma}$ and $\mathrm{T}_{\mu_kQ,\sigma}$.
We sometimes identify them with each other and write simply $\mathrm{T}_{\sigma}$ since we do not want to specify a choice of a quiver.

\section{Tilting of t-structures}\label{sec_2}
\subsection{Torsion pair and tilting}
Let $\mca{D}$ be a triangulated category and $\mca{A}$ be the core of a t-structure. 
\begin{defn}
A pair $(\mca{T},\mca{F})$ of full subcategories of $\mca{A}$ is called a {\it torsion pair} if the following conditions are satisfied :  
\begin{itemize}
\item[(TP1)] 
for any $T\in\mca{T}$ and any $F\in\mca{F}$, we have $\Hom(T,F)=0$, 
\item[(TP2)] 
for any $X\in \mca{A}$, there exists an exact sequence 
\[
0\to T\to X\to F\to 0
\]
with $T\in\mca{T}$ and $F\in\mca{F}$. 
\end{itemize}
\end{defn}
We sometimes illustrate the torsion pair as in Figure \ref{fig_TP}. 
In the figure, we have no non-trivial morphism from an object on left to an object on right.
\begin{figure}[htbp]
  \centering
  \input{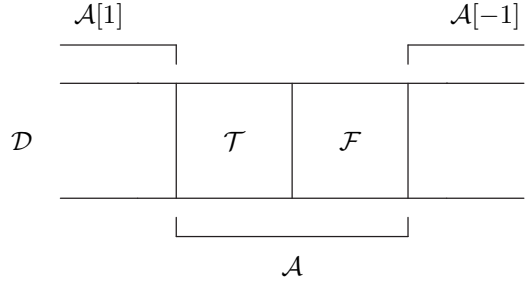}
  \caption{Torsion pair}
  \label{fig_TP}
\end{figure}

Given a torsion pair $(\mca{T},\mca{F})$, 
let $\mca{D}_{\leq -1}^{(\mca{F},\mca{T}[-1])}$ denote the full subcategory of $\mca{D}$ consisting of objects $E$ which satisfy
\[
H^i_{\mca{A}}(E)
\begin{cases}
\in \mca{T}& i=0,\\
= 0 & i \geq 1,
\end{cases}
\]
and let $\mca{D}_{\geq 0}^{(\mca{F},\mca{T}[-1])}$ denote the full subcategory of $\mca{D}$ consisting of objects $E$ which satisfy
\[
H^i_{\mca{A}}(E)
\begin{cases}
\in \mca{F} & i=0,\\
= 0 & i \leq -1.
\end{cases}
\]
Then the pair of full subcategories
\[
\left(\mca{D}_{\leq -1}^{(\mca{F},\mca{T}[-1])},\mca{D}_{\geq 0}^{(\mca{F},\mca{T}[-1])}\right)
\]
gives a t-structure of $\mca{D}$ (see Figure \ref{fig_TILTING}).
\begin{figure}[htbp]
  \centering
  \input{fig_tilting.tpc}
  \caption{Tilting with respect to $(\mca{T},\mca{F})$}
  \label{fig_TILTING}
\end{figure}
Let 
\[
\mca{A}^{(\mca{F},\mca{T}[-1])}:=
\mca{D}_{\leq -1}^{(\mca{F},\mca{T}[-1])}[-1]
\cap 
\mca{D}_{\geq 0}^{(\mca{F},\mca{T}[-1])}
\]
be the heart of the t-structure. 
That is, $\mca{A}^{(\mca{F},\mca{T}[-1])}$ is 
the full subcategory of $\mca{D}$ consisting of objects $E$ which satisfy
\[
H^i_{\mca{A}}(E)
\begin{cases}
\in \mca{F} & i=0,\\
\in \mca{T} & i=1,\\
= 0 & i \neq 0,1.
\end{cases}
\]

\subsection{Composition of tilting}
Let $(\mca{T},\mca{F})$ be a torsion pair of $\mca{A}$ and we put $\mca{A}':=\mca{A}^{(\mca{F},\mca{T}[-1])}$.
Let $(\mca{T}',\mca{F}')$ be a torsion pair of $\mca{A}'$  such that $\mca{T}'\subset \mca{F}$. We put $\mca{A}''=(\mca{A}')^{(\mca{F}',\mca{T}'[-1])}$.

Let $\mca{T}''$ denote the full subcategory of $\mca{A}$ consisting of $X$ with $F_X\in \mca{T}'$ where $F_X\in \mca{F}$ is the quotient object of $X$ associated to the exact sequence (TP2) for $(\mca{T},\mca{F})$.
Let $\mca{F}''$ denote the full subcategory of $\mca{F}$ consisting of $Y$ with $T_Y=0$ where $T_Y\in \mca{T}'$ is the subobject of $Y$ associated to the exact sequence (TP2) for $(\mca{T}',\mca{F}')$. (See Figure \ref{fig_lem23}.)
We can easily verify the following lemma.
\begin{figure}[htbp]
  \centering
  \input{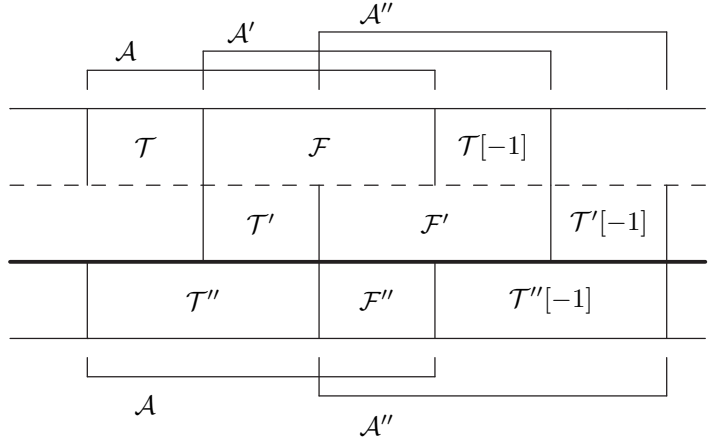}
  \caption{Composition of tilting (Lemma \ref{lem_23})}
  \label{fig_lem23}
\end{figure}
\begin{lem}\label{lem_23}
The pair of the full subcategories $(\mca{T}'',\mca{F}'')$ gives a torsion pair of $\mca{A}$ and 
\[
\mca{A}''
=
\mca{A}^{(\mca{F}'',\mca{T}''[-1])}
\]
\end{lem}
On the other hand, assume that $(\mca{T}',\mca{F}')$ is a torsion pair of $\mca{A}'$  such that $\mca{F}'\subset \mca{T}[-1]$.
Let $\mca{F}''$ denote the full subcategory of $\mca{A}$ consisting of $Y$ with $T_Y\in \mca{F}'[1]$ where $T_Y\in \mca{T}$ is the subobject of $Y$ associated to the exact sequence (TP2) for $(\mca{T},\mca{F})$.
Let $\mca{T}''$ denote the full subcategory of $\mca{T}$ consisting of $X$ with $F_X=0$ where $F_X\in \mca{F}'[1]$ is the quotient object of $X$ associated to the exact sequence (TP2) for $(\mca{T}'[1],\mca{F}'[1])$ (see Figure \ref{fig_lem24}). 
We put $\mca{A}'':=(\mca{A}')^{(\mca{F}'[1],\mca{T}')}$
We can also verify the following lemma. 
\begin{figure}[htbp]
  \centering
  \input{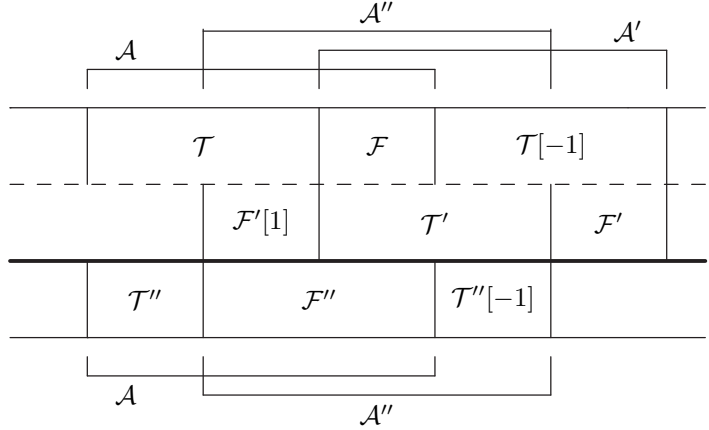}
  \caption{Composition of tilting (Lemma \ref{lem_24})}
  \label{fig_lem24}
\end{figure}
\begin{lem}\label{lem_24}
The pair of the full subcategories $(\mca{T}'',\mca{F}'')$ gives a torsion pair of $\mca{A}$ and 
\[
\mca{A}'':=\mca{A}^{(\mca{F}'',\mca{T}''[-1])}.
\]
\end{lem}

\subsection{Mutation and tilting}\label{sec_}
Let $\mca{S}_k$ be the full subcategory consisting of $J_{Q,W}$-modules supported on the vertex $k$.
We put
\begin{align*}
(\mca{S}_k)^{\bot}&:=
\{
E\in \mr{Mod}J_{Q,W}\mid \mr{Hom}(s_k,E)=0\},\\
{}^{\bot}(\mca{S}_k)&:=
\{
E\in \mr{Mod}J_{Q,W}\mid \mr{Hom}(E,s_k)=0
\}.
\end{align*}
Then both $(\mca{S}_k,(\mca{S}_k)^{\bot})$ and $({}^{\bot}(\mca{S}_k),\mca{S}_k)$ give torsion pairs of $\mr{Mod}J_{Q,W}$.
It is shown in \cite[Corollary 5.5]{dong-keller} that the derived equivalences associated to a mutation are given by tilting with respect to these torsion pairs :
\begin{align}
\Phi_{k,+}^{-1}\left(\mr{Mod}J_{\mu_k(Q,W)}\right)
& = 
(\mr{Mod}J_{Q,W})^{\left((\mca{S}_k)^{\bot},\mca{S}_k[-1]\right)},\label{eq_tilting+}\\
\Phi_{k,-}^{-1}\left(\mr{Mod}J_{\mu_k(Q,W)}\right)
& = 
(\mr{Mod}J_{Q,W})^{\left(\mca{S}_k[1],{}^\bot(\mca{S}_k)\right)}.\label{eq_tilting-}
\end{align}

\subsection{Composition of mutations and tilting}\label{subsec_comp}
The proof of the following theorem is due to Tom Bridgeland.
We put $\bar{\mca{A}}:=\mr{Mod}J_{Q,W}$. 
\begin{thm}\label{thm_tilting}
There exists a unique sequence $\ep(1),\ldots, \ep(l)$ of signs which satisfies the following conditions; We put
\[
\Phi_{\kk}:=
\Phi_{k_l,\ep(l)}\circ\cdots\circ\Phi_{k_1,\ep(1)}\colon \mca{D}\Gamma\overset{\sim}{\longrightarrow} \mca{D}\Gamma_{\mu_\kk(Q,W)}
\]
and
\[
\bar{\mca{A}}_{\kk}:=\Phi_{\kk}^{-1}(\mr{Mod}J_{\mu_\kk(Q,W)}).
\]
Then 
\begin{itemize}
\item[$(A_l)$]there exists a torsion pair $(\bar{\mca{T}}_{\kk},\bar{\mca{F}}_{\kk})$ of $\bar{\mca{A}}$ such that 
\[
\bar{\mca{A}}^{(\bar{\mca{F}}_{\kk},\bar{\mca{T}}_{\kk}[-1])}=\bar{\mca{A}}_{\kk}.
\]
\item[$(B_l)$]$\Phi_{\kk}^{-1}(s_{\kk,i})\in \bar{\mca{F}}_{\kk}$ or $\Phi_{\kk}^{-1}(s_{\kk,i})\in \bar{\mca{T}}_{\kk}[-1]$ for any $i\in Q_0$ where $s_{\kk,i}$ is the simple $J_{\mu_{\kk}(Q,W)}$-module.
\end{itemize}
\end{thm}
\begin{proof}
We prove the claim by induction with respect to the length $l$ of the sequence.
First of all, $(A_1)$ is hold if we take $\ep(1)=+$ by \eqref{eq_tilting+}.

\vspace{1mm}

\noindent \underline{$(A_l)\Longrightarrow (B_l)$} :

\vspace{2mm}

Since $(\bar{\mca{F}}_{\kk},\bar{\mca{T}}_{\kk}[-1])$ give a torsion pair for $\bar{\mca{A}}_{\kk}$, an exact sequence is associated to $\Phi_{\kk}^{-1}(s_{\kk,i})$.
Because $\Phi_{\kk}^{-1}(s_{\kk,i})$ is simple in $\bar{\mca{A}}_{\kk}$, we have
\[
\Phi_{\kk}^{-1}(s_{\kk,i})\in \bar{\mca{F}}_{\kk}
\]
or
\[
\Phi_{\kk}^{-1}(s_{\kk,i})\in \bar{\mca{T}}_{\kk}[-1].
\]

\vspace{1mm}

\noindent \underline{$(B_l)\Longrightarrow (A_{l+1})$} :

\vspace{2mm}

We define $\ep(l)$ by 
\[
\ep(l)=
\begin{cases}
+ & \text{if }\Phi_{\kk}^{-1}(s_{\kk,i})\in \bar{\mca{F}}_{\kk},\\
- & \text{if }\Phi_{\kk}^{-1}(s_{\kk,i})\in \bar{\mca{T}}_{\kk}[-1].\\
\end{cases}
\]
Then the claim follows Lemma \ref{lem_23} and Lemma \ref{lem_24}.
\end{proof}
\begin{rem}
A similar statement has been shown in \cite[Theorem 2.15]{plamondon}.
\end{rem}
For $1\leq r\leq l$, let $\kk^{(r)}$ denote the truncated sequence $(k_1,\ldots,k_r)$.
For $i\in I$, we define $s^{(r)}_{i}\in \bar{\mca{A}}$
\[
s^{(r)}_{i}:=
\begin{cases}
\Phi_{\kk^{(r)}}^{-1}(s_{\kk^{(r)},i}) & \text{if } \Phi_{\kk^{(r)}}^{-1}(s_{\kk^{(r)},i})\in \bar{\mca{F}}_{\kk^{(r)}},\\
\Phi_{\kk^{(r)}}^{-1}(s_{\kk^{(r)},i})[1] & \text{if } \Phi_{\kk^{(r)}}^{-1}(s_{\kk^{(r)},i})\in \bar{\mca{T}}_{\kk^{(r)}}[-1].\\
\end{cases}
\]
We put $s^{(r)}:=s^{(r)}_{k_r}$.

The canonical t-structure of $\da$ induces a t-structure of $\dfa$ whose core is $\mca{A}:=\mr{mod}J_{Q,W}$. 
Since $s^{(r)}\in \mca{A}$ for any $r$, we have $\bar{\mca{T}}_\kk\in \mca{A}$. 
We put 
\[
\mca{A}_\kk:=\Phi_{\kk}^{-1}(\mr{mod}J_{\mu_\kk(Q,W)}),\quad 
\mca{T}_\kk:=\bar{\mca{T}}_\kk,\quad
\mca{F}_\kk:=\bar{\mca{F}}_\kk\cap\mca{A}.
\]
Then we can verify the following :
\begin{cor}
The pair of the full subcategories $(\mca{T}_\kk,\mca{F}_\kk)$ gives a torsion pair of $\mca{A}$ and 
\[
\mca{A}_\kk
=
\mca{A}^{(\mca{F}_\kk,\mca{T}_\kk[-1])}.
\]
\end{cor}

\begin{figure}[htbp]
  \centering
  \input{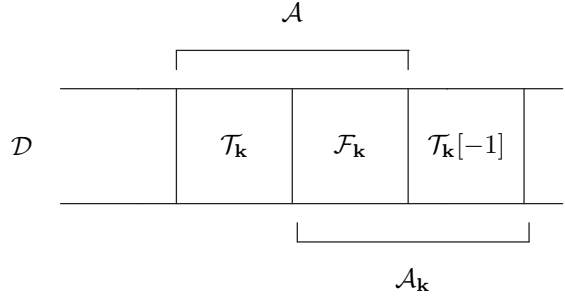}
  \caption{Composition of mutation and tilting}
  \label{}
\end{figure}

\section{Stability condition on $\dfa$}\label{subsec_stab}
In this section, we study the space of stability conditions on $\dfa$.
For a subcategory $\mca{C}\subset \dfa$,
let $C_{\mca{C}}\subset L_\R$ be the minimal cone containing all the classes of elements in $\mca{C}$ and we define its dual cone $C_{\mca{C}}^*$ by
\[
C_{\mca{C}}^*:=\bigl\{\theta\in (L_\R)^*=M_\R \mid \langle\theta,\mathbf{v}\rangle>0\text{ for any }\mathbf{v}\in C_\mca{C}\bigr\}.
\]

Throughout this section, we fix an element $\delta\in \mca{C}^*_\A$.
For $\theta\in (L_\R)^*=M_\R$, let
\[
Z_\theta\colon L\to \C
\]
denote the group homomorphism given by $Z_\theta:=\langle-\delta+\sqrt{-1}\theta,\bullet\rangle$.

\subsection{Embedding of $M_\R$}\label{subsec_embed}
If $\theta\in C_\A^*$, the pair $\zeta(\theta):=(\mca{A},Z_\theta)$
gives a Bridgeland's stability condition on $\dfa$.
This gives an embedding 
\[
\zeta \colon C_\A^* \hookrightarrow \mathrm{Stab}(\dfa)
\]
where the right hand side is the space of Bridgeland stability conditions on $\dfa$. 
We will extend this to an embedding of $(L_\R)^*=M_\R$.

For two real numbers $t$ and $\phi$, we define $t^*\phi\in \R$ so that 
\[
\mr{tan}((t^*\phi)\pi)=\mr{tan}(\phi\pi)+t, \quad 0^*\phi=\phi
\]
and so that the map $(t,\phi)\mapsto t^*\phi$ is continuous.
For $\theta\in C_\A^*$, let $\mca{P}_\theta$ the slicing of $\dfa$ corresponding to the stability condition $\zeta(\theta)$ (\cite[Definition 5.1 and Proposition 5.3]{bridgeland-stability}).
That is, $\mca{P}_\theta(\phi)$ is the full subcategory of semistable objects with phase $\phi\in \R$ with respect to the stability condition $\zeta(\theta)$. 
We define the slicing $t^*\mca{P}_\theta$ by
\[
t^*\mca{P}_\theta(\phi):=\mca{P}_\theta(t^*\phi).
\]
Then the pair $(t^*\mca{P}_\theta,Z_{\theta-t\delta})$ gives a stability condition (Figure \ref{fig_stability}). 
\begin{figure}[htbp]
\hspace{-25mm}
  \input{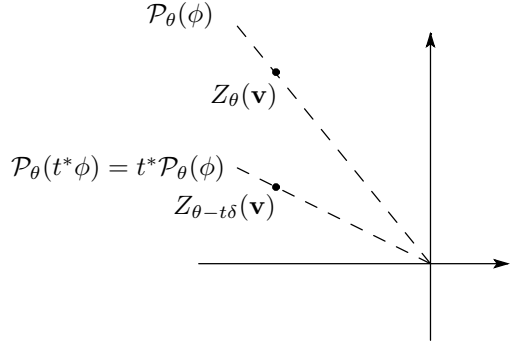}
  \caption{${\zeta}(\theta-t\delta)$}
  \label{fig_stability}
\end{figure}
We define the map
\[
\zeta \colon (L_\R)^*=M_\R \to \mathrm{Stab}(\dfa).
\]
by 
\[
{\zeta}(\theta-t\delta):=(t^*\mca{P}_\theta,Z_{\theta-t\delta})
\]
for any $\theta\in C_\A^*$ and $t\in \mathbb{R}$. 
We can verify that this is well-defined and injective.


\subsection{T-structures}\label{subsec_Tstr}
In this subsection, we describe the t-structures corresponding to some stability conditions in $\zeta(M_\R)$.
For a stability condition $\zeta$, let $\mathcal{A}_\zeta$ denote the core of the t-structure corresponding to $\zeta$, i.e. the full subcategory of objects whose HN factors have phases in $[0,1)$. 
\begin{prop}\label{prop_27}
For $\theta\in C_{\A_\kk}^*$, we have $\mathcal{A}_{\zeta(\theta)}=\mathcal{A}_{\kk}$.
\end{prop}
\begin{proof}
We will prove by induction with respect to the length $l$ of the sequence $\kk$. 
For $1\leq r\leq l$ 
we put $\mca{A}^{(r)}:=\mca{A}_{\kk^{(r)}}$, where $\kk^{(r)}$ is the truncated sequence.

For $i\in I$, 
let $W^{(r-1)}_{i}$ denote the hyperplane which is perpendicular to $s^{(r)}_{i}$:
\[
W^{(r-1)}_{i}:=
\Bigl\{\theta\,\Big|\, \bigl\langle\theta,[s^{(r)}_{i}]\bigr\rangle = 0 \Bigr\}.
\]
Note that the boundary of $C_{\A^{{(r-1)}}}^*$ is contained in the union of $W^{(r-1)}_{i}$'s.

Assume that $\mathcal{A}_{\zeta(\theta)}=\A^{(r-1)}$ for $\theta\in C_{\A^{(r-1)}}^*$.
Take $\theta'\in C_{\A_{(r)}}^*$ which is sufficiently close to the hyperplane $W^{(r-1)}_{k_r}=W^{(r)}_{k_r}$ and which is sufficiently far from the other hyperplanes.
It is enough to show that $\A_{\zeta(\theta')}=\A^{(r)}$.

In the case of $\ep(r)=+$, we have
\[
\mr{Re}Z_{\theta}(s(r)),\mr{Re}Z_{\theta'}(s(r))<0,\quad 
\mr{Im}Z_{\theta}(s(r))>0,\quad \mr{Im}Z_{\theta'}(s(r))<0
\]
(see Figure \ref{fig_WC}).
\begin{figure}[htbp]
  \centering
  \input{WC+.tpc}
  \caption{the case of $\ep(r)=+$.}
  \label{fig_WC}
\end{figure}So the core $\A_{\zeta(\theta')}$ is given by tilting the core $\A_{\zeta(\theta)}$ with respect to the torsion pair
\[
\left(\mca{S}{(r)},(\mca{S}{(r)})^\bot\right)
\]
where 
\[
\mca{S}{(r)}:=\bigl\{\bigl(s^{(r)}\bigr)^{\oplus n}\mid n\geq 0\bigr\}.
\]
By \eqref{eq_tilting-}, we get $\A_{\zeta(\theta')}=\A^{(r)}$. We can see in the case of $\ep(r)=-$ in the same way.
\end{proof}
\begin{thm}
Assume we have ${C}_\kk^*={C}_{\kk'}^*$. 
Then, the equivalence $\Phi_{\kk'}\circ\Phi_\kk^{-1}$ induces an equivalence from $\mr{mod}J_{\kk}$ to $\mr{mod}J_{\kk'}$.
Moreover, there is a unique permutation $\kappa\in \mathfrak{S}_I$ of $I$ such that 
\[
\Phi_{\kk'}\circ\Phi_\kk^{-1}(s_{\kk,i})=s_{\kk',\kappa(i)}
\]
\end{thm}
\begin{proof}
The equivalence for $\mr{mod}J_{\kk}$ is a consequence of Proposition \ref{prop_27}. 
The permutation is induced by the description of the boundary of the chamber.
\end{proof}

\section{Statements}\label{sec_statement}
\subsection{Quiver Grassmannian}\label{subsec_grass}
Let 
$\Gamma_{\kk,i}$ denote the direct summand of $\Gamma_\kk$ and 
$P_{\kk,i}$ denote the projective indecomposable $J_\kk$-module. 
We put
\[
R_{\kk,i}:=H^{1}_{\bar{\A}}(\Phi_{\kk}^{-1}(\Gamma_{\kk,i}))=H^{1}_{\bar{\A}}(\Phi_{\kk}^{-1}(P_{\kk,i}))\in \mca{T}_\kk\subset \A.
\]
\begin{defn}
For $\mathbf{v}\in L$, let $\mathrm{Grass}(\mathbf{k};i,\mathbf{v})$ be the moduli scheme which parametrizes elements $V$ in $\A$ equipped with surjections from $R_{\kk,i}$ such that $[V]=\mathbf{v}$\textup{:}
\[
\mathrm{Grass}(\mathbf{k};i,\mathbf{v}):=\{R_{\kk,i}\twoheadrightarrow V\mid V\in \A,\ [V]=\mathbf{v}\}.
\]
\end{defn}
We call $\mathrm{Grass}(\mathbf{k};i,\mathbf{v})$ as a {\it quiver Grassmannian}.
\begin{rem}
We can construct the moduli scheme as a GIT quotient \textup{(see \cite[\S 5.1]{ncdt-vo})}.
\end{rem}
Let $\mca{M}_\A$ be the moduli stack of objects in $\A$ and $\nu_{\mca{M}_\A}$ be the Behrend function on it (\cite{behrend-dt}).
We define 
\begin{align*}
e_+(\mathrm{Grass}(\mathbf{k};i,\mathbf{v}))&:=e(\mathrm{Grass}(\mathbf{k};i,\mathbf{v})),\\
e_-(\mathrm{Grass}(\mathbf{k};i,\mathbf{v}))&:=
\int_{\mathrm{Grass}(\mathbf{k};i,\mathbf{v})}\pi^*(\nu_{\mca{M}_\A})\cdot\mr{d}e\\
&=\sum_{n\in \Z} n\cdot e(\pi^*(\nu_{\mca{M}_\A})^{-1}(n)),
\end{align*}
where $\pi\colon \mathrm{Grass}(\mathbf{k};i,\mathbf{v})\to \mca{M}_\A$ is the forgetful morphism and $e(-)$ represents the topological Euler characteristics.

\subsection{On non-commutative Donaldson-Thomas invariants}\label{subsec_ncdt}
\subsubsection{Non-commutative Donaldson-Thomas invariants}\label{subsub_ncdt}
For a vertex $i\in I$ and an element $\mathbf{v}\in L$, let $\mr{Hilb}_J(i;\mathbf{v})$ be the moduli scheme which parametrizes elements in $V\in \moda$ equipped with a surjection from $P_i$ such that $[V]=\mathbf{v}$:
\[
\mr{Hilb}_J(i;\mathbf{v}):=\{P_i\twoheadrightarrow V\mid V\in \A, [V]=\mathrm{v} \}.
\]
\begin{defn}[\protect{\cite{szendroi-ncdt}}]\label{defn_ncdt}
We define invariants by
\begin{align*}
&\mr{DT}_{J,+}(i;\mathbf{v}):=e({\mr{Hilb}_J(i;\mathbf{v})}),\\
&\mr{DT}_{J,-}(i;\mathbf{v}):=
\int_{\mr{Hilb}_J(i;\mathbf{v})}\pi^*(\nu_{\mca{M}_\A})\cdot \mr{d}e=\sum_{n\in Z}n\cdot e(\pi^*(\nu_{\mca{M}_\A})^{-1}(n)),
\end{align*}
where $\pi\colon {\mr{Hilb}_J(i;\mathbf{v})}\to \mca{M}_\A$ is the forgetful morphism.
\end{defn}
\begin{rem}
The non-commutative Donaldson-Thomas invariants in \cite{szendroi-ncdt} are defined using the Behrend function on $\mr{Hilb}_J(i;\mathbf{v})$\textup{:}
\[
\mr{DT}_{J,\mr{Sze}}(i;\mathbf{v}):=
\int_{\mr{Hilb}_J(i;\mathbf{v})}\nu_{\mr{Hilb}_J(i;\mathbf{v})}\cdot \mr{d}e:=\sum_{n\in Z}n\cdot e(\nu_{\mr{Hilb}_J(i;\mathbf{v})}^{-1}(n)).
\]
We have $\pi^*(\nu_{\mca{M}_\A})=(-1)^{v_i}\cdot\nu_{\mr{Hilb}_J(i;\mathbf{v})}$ and 
\[
\mr{DT}_{J,\mr{Sze}}(i;\mathbf{v})=(-1)^{v_i}\cdot\mr{DT}_{J,-}(i;\mathbf{v}).
\]
\end{rem}
We define generating functions by
\[
\mca{Z}_{J,\sigma}^{i}:=\sum_{\mathbf{v}}\mr{DT}_{J,\sigma}(i;\mathbf{v})\cdot \mathbf{y}^{\mathbf{v}}_\sigma.
\]

\subsubsection{Torus automorphism via ncDT invariants}\label{subsub_auto}
For a full subcategory $\mca{C}\subset\dfa$, we define $\hT_{\mca{C},\sigma}$ 
and $\hDT_{\mca{C},\sigma}$ by
\begin{align*}
&\hT_{\mca{C},\sigma}:=\biggl(\,\prod_{\mathbf{v}\in L\cap C_\mca{C}} \C\cdot \mathbf{y}_\sigma^\mathbf{v}\biggr) \oplus \biggl(\,\bigoplus_{\mathbf{v}\in L-C_\mca{C}} \C\cdot \mathbf{y}_\sigma^\mathbf{v}\biggr),\\
&\hDT_{\mca{C},\sigma}:=\T_{\mca{C},\sigma}^\vee\otimes \hT_{\mca{C},\sigma}
\end{align*}
where $C_\mca{C}\subset L_\R$ is the minimal cone which contains all the classes of elements in $\mca{C}$.
They are called {\it the completions with respect to $\mca{C}$}. 
If $\mca{C}$ is a subcategory of a core of a t-structure, then the products extend to these completions.
Moreover, if $\mca{C'}\subset\mca{C}$ then the completions with respect to $\mca{C'}$ give subalgebras of the ones with respect to $\mca{C}$.
If an automorphism of the completion with respect to $\mca{C}'$ lifts to the one for $\mca{C}$, we use the same symbol as the original automorphism for the lift.
Note that $\mca{Z}_{J,\sigma}^{i}$ gives an elements in $\hT_{\A,\sigma}$.
\begin{defn}
We define torus automorphisms
\[
\mca{DT}_{J,\sigma}\colon \hDT_{\A,\sigma} \overset{\sim}{\too} \hDT_{\A,\sigma}
\]
by
\[
\mca{DT}_{J,\sigma}(x_{i,\sigma}):=x_{i,\sigma}\cdot \mca{Z}_{J,\sigma}^{i},\quad 
\mca{DT}_{J,\sigma}(y_{i,\sigma}):=y_{i,\sigma}\cdot \prod_j(\mca{Z}_{J,\sigma}^{j})^{\bar{Q}(j,i)}.
\]
\end{defn}

\subsubsection{Transformation formula of ncDT invariants}\label{subsub_trans}
\begin{defn}
We define torus automorphisms
\[
\mr{Ad}_{\mca{T}_\kk[-1],\sigma}\colon \DT_{\sigma} \overset{\sim}{\too} \DT_{\sigma}
\]
by
\begin{align}
\mr{Ad}_{\mca{T}_\kk[-1],\sigma}(x_{\kk,i,\sigma})
&:=x_{\kk,i,\sigma}\cdot \Biggl(\sum_\mathbf{v}e_\sigma\Bigr(\mathrm{Grass}(\mathbf{k};i,\mathbf{v})\Bigr)\cdot\mathbf{y}_\sigma^{-\mathbf{v}}\Biggr),\\
\mr{Ad}_{\mca{T}_\kk[-1],\sigma}(y_{\kk,i,\sigma})
&:=y_{\kk,i,\sigma}\cdot \prod_j\Biggl(\sum_\mathbf{v}e_\sigma\Bigr(\mathrm{Grass}(\mathbf{k};j,\mathbf{v})\Bigr)\cdot\mathbf{y}_\sigma^{-\mathbf{v}}\Biggr)^{\bar{Q}(j,i)}.
\end{align}
They lift to the completions with respect to $\A_\kk$. 
We also define $\mr{Ad}_{\mca{T}_\kk,\sigma}$ by 
\[
\mr{Ad}_{\mca{T}_\kk,\sigma}
:=\Sigma\circ \mr{Ad}_{\mca{T}_\kk[-1],\sigma}\circ \Sigma
\]
which lift to the completions with respect to $\A$. \textup{(}See \eqref{eq_Sigma} for the definition of $\Sigma$.\textup{)} 
\end{defn}
\begin{thm}\label{thm_trans}
The composition 
\[
\mr{Ad}_{\mca{T}_\kk,\sigma}^{-1}\circ \mca{DT}_{J,\sigma}
\]
preserves $\hDT_{\mca{F}_\kk,\sigma}$ and lifts to the automorphism of $\hDT_{\A_\kk,\sigma}$.
Moreover, we have the following identity of automorphisms of $\hDT_{\A_\kk,\sigma}$\textup{:}
\[
\mca{DT}_{J_\kk,\sigma}
=
\mr{Ad}_{\mca{T}_\kk,\sigma}^{-1}\circ \mca{DT}_{J,\sigma}\circ \mr{Ad}_{\mca{T}_\kk[-1],\sigma}.
\]
\end{thm}

\subsection{Caldero-Chapoton formula}
In this subsection, we put $\sigma=+$ and use notations without ``+''.
We identify $\C(x_1,\ldots,x_n)$ with the fractional field of $T_Q$. 
\begin{thm}\label{thm_CC}
We have
\begin{equation}\label{eq_CC}
\mathrm{FZ}_{\kk,i}(\underline{x})=
x_{\kk,i}\cdot \Biggl(\sum_\mathbf{v}e\Bigr(\mathrm{Grass}(\mathbf{k};i,\mathbf{v})\Bigr)\cdot\mathbf{y}^{-\mathbf{v}}\Biggr).
\end{equation}
where $(\underline{y})^{-\mathbf{v}}=\prod_j (y_j)^{-v_j}$ and $y_j=\prod_i(x_i)^{\bar{Q}(i,j)}$
\end{thm}

\begin{ex}\label{ex_cluster}
If we take a sequence $\kk=(k)$ of length $1$, then we have
\[
R_{(k),i}=\begin{cases}
0 & i\neq k,\\
s_k & i=k.
\end{cases}
\]
Hence we have
\begin{equation}\label{eq_preCT}
FZ_{i,k}(\underline{x})=
\begin{cases}
x_i' & i\neq k,\\
x_k'(1+(y_k)^{-1}) & i=k,
\end{cases}
\end{equation}
where $y_k=\prod_j (x_j)^{\bar{Q}(j,k)}$, $x_i':=\mathbf{x}^{[\Gamma'_i]}$ and $\Gamma'_i$ is the direct summand of $\mu_k\Gamma$.
Note that by \eqref{eq_yy'} we have
\[
x_i'=x_i\ (i\neq k),\quad 
x_k'=(x_k)^{-1}\prod_j (x_j)^{Q(j,k)},\quad
\]
Substituting these for \eqref{eq_preCT}, we get the cluster transformation \eqref{eq_cluster}.
\end{ex}

\section{Review: Motivic Hall algebra}\label{sec_MH}
\subsection{Motivic Hall algebra and its limit}
\subsubsection{Relative Grothendieck ring of stacks}
For an algebraic stack $\sS$, let $\mathrm{St}/\sS$ denote the category whose objects are finite type stacks $\X$ over $\C$ equipped with a morphism to $\sS$.

A morphism of stacks $f\colon \to Y$ is said to be a 
{\it geometric bijection} if it is representable and the induced functor on groupoids of $\C$-valued points
\[
f(\C)\colon X(\C) \to Y(\C)
\]
is an equivalence of categories (\cite[Definition 3.1]{bridgeland-hall}).

A morphism of stacks $f\colon \to Y$ is said to be a {\it Zariski fibration} if its pullback
to any scheme is a Zariski fibration of schemes (\cite[Definition 3.3]{bridgeland-hall}).

We define $K(\mathrm{St}/\sS)$ by the free Abelian group spanned by isomorphism classes of $\mathrm{St}/\sS$ modulo the following relations:
\begin{enumerate}
\item[(1)] 
$[\X_1\sqcup \X_2\overset{f_1\sqcup f_2}{\too}\sS]=
[\X_1\overset{f_1}{\too}\sS]+
[\X_2\overset{f_2}{\too}\sS]$,
\item[(2)] 
$[\X_1\overset{f_1}{\too}\sS]=[\X_2\overset{f_2}{\too}\sS]$ if there is a geometric bijection $g\colon\X_1\to\X_2$ with $f_1=f_2\circ g$,
\item[(3)] 
$[\X_1\overset{f_1}{\too}\sS]=[\X_2\overset{f_2}{\too}\sS]$ if there is a factorisations $f_i=g\circ h_i$ such that $h_i\colon \X_i\to \Y$ are Zariski fibrations with the same fibres
\end{enumerate}
(\cite[Definition 3.6]{bridgeland-hall}). We call $K(\mathrm{St}/\sS)$ as the {\it relative Grothendieck ring of stacks over $\sS$}.

A morphism of stacks $\psi\colon\mathcal{T}\to\mathcal{S}$ induces a map 
\[
\psi_*\colon K(\mr{St}/\tT)\to K(\mr{St}/\sS)
\]
sending $[g\colon \Y\to\tT]$ to $[\psi \circ g\colon \X\to\sS]$.
If $\psi$ is of finite type it also induces a map 
\[
\psi^*\colon K(\mr{St}/\sS)\to K(\mr{St}/\tT)
\]
sending $[f\colon \X\to\sS]$ to the map $[g\colon \Y\to\tT]$ in the following Cartesian diagram:
\[
\xymatrix{
\mathcal{Y} \ar[d] \ar[r]^g \ar@{}[dr]|\square & \tT \ar[d]^{\psi}  & \\
\mathcal{X} \ar[r]_f & \sS \ar@{}[r]_. &
}
\]

\subsubsection{Motivic Hall algebra}
Let $\M$ be the moduli stack of all objects in $\Afd=\mathrm{mod}J$ and $\MM$ be the moduli stack of all exact sequences in $\Afd$. 
Let $p_\varepsilon\colon \MM\to\M$ ($\varepsilon=1,2,3$) be the morphism given by taking $\varepsilon$-th terms of exact sequences.  
Then, $p_2$ is of finite type.
Using the diagram
\[
\xymatrix{
\MM \ar[d]_{p_1\times p_3} \ar[r]^{p_2} & \M \\
\M\times\M & 
}
\]
we define a product $*$ on $K(\mathrm{St}/\M)$ by
\begin{equation}\label{eq_hall}
*:=(p_2)_*(p_1\times p_3)^*\colon K(\mathrm{St}/\M)\otimes K(\mathrm{St}/\M)\to K(\mathrm{St}/\M).
\end{equation}
We put $\mathrm{MH}(\Afd):=K(\mathrm{St}/\M)$. 
The algebra $(\mathrm{MH}(\Afd),*)$ is called the {\it motivic Hall algebra} of the Abelian category $\Afd$.
\begin{thm}[\protect{\cite{joyce-2},\cite[Theorem 4.1]{bridgeland-hall}}]
The motivic Hall algebra $(\mathrm{MH}(\Afd),*)$ is associative.
\end{thm}
\begin{rem}
The $3$-Calabi-Yau property of the category $\A$ is not necessary for this theorem.
\end{rem}

\subsubsection{Semi-classical limit of the motivic Hall algebra}
Let $\mr{MH}_0(\A)\subset\mr{MH}(\A)$ be the $K(\mr{Var}/\C)[\mathbb{L}^{-1}]$-submodule generated by classes 
\[
[X\overset{f}{\too}\mca{M}_{\A}]
\]
with $X$ a variety.
\begin{thm}[\protect{\cite[Theorem 5.2]{bridgeland-hall}}]
\begin{itemize}
\item[\textup{(1)}] $\mathrm{MH}_0(\A)\subset \mathrm{MH}(\A)$ is a subring.
\item[\textup{(2)}] The product induced on the quotient
\[
\mathrm{MH}_{\mr{sc}}(\A):=\mathrm{MH}_{0}(\A)/(\mathbb{L}-1)\mathrm{MH}_{0}(\A)
\]
is commutative $K(\mr{Var}/\C)$-algebra.
\end{itemize}
\end{thm}
We define a Poisson bracket $\{-,-\}$ on $\mathrm{MH}_{\mr{sc}}(\A)$ by 
\[
\{f,g\}=\frac{f*g-g*f}{\mathbb{L}-1}\quad \mathrm{mod}(\mathbb{L}-1).
\]

\subsubsection{Completion of the motivic Hall algebra}
Note that the moduli stack has the canonical decomposition 
\[
\M=\bigsqcup_{\mathbf{v}\in C_{\A}\cap L} \mathcal{M}_{\A}(\mathbf{v}).
\]
We put
\[
\hMH(\Afd):=\prod_{\mathbf{v}\in C_{\Afd}\cap L} K(\mr{St}/\M(\mathbf{v})),
\]
then the $*$-product canonically extends to $\hMH(\Afd)$.

Let $\mca{C}\subset\mca{A}$ be an extension closed full subcategory.
Assume that the moduli stack $\mca{M}_{\mca{C}}\subset\mca{M}_{\A}$ of objects in $\mca{C}$ is algebraic.
Let $\mathrm{MH}(\mca{C})$ denote the subalgebra consisting of the elements $[f\colon \mca{X}\to\mca{M}_{\A}]$ such that $f$ factors through $\mca{M}_{\mca{C}}\subset\mca{M}_{\A}$.
We put
\[
\hMH(\mca{C}):=\prod_{\mathbf{v}\in C_\mca{C}\cap L}\Bigl(\MH(\mca{C})\cap K(\mr{St}/\M(\mathbf{v}))\Bigr).
\]
We define $\hMH_{0}(\A)$, $\hMH_{\mr{sc}}(\A)$ $\hMH_{0}(\mca{C})$ and $\hMH_{\mr{sc}}(\mca{C})$ in the same way.

\subsection{Quantum torus and integration map}
\subsubsection{Quantum torus and semi-classical limit}
The {\it quantum torus}, {\it dual quantum torus} and {\it double quantum torus}  for $\dfa$ is the $\C(t)$-vector spaces
\[
\mathrm{QT}:=\sum_{\mathbf{v}\in L}\C(t)\cdot \mathbf{y}^{\mathbf{v}}, \quad
\mathrm{QT}^\vee:=\sum_{\mathbf{w}\in M}\C(t)\cdot \mathbf{x}^{\mathbf{w}}, \quad
\mathrm{Q}\mathbb{T}:=\mathrm{QT}^\vee\otimes\mathrm{QT}
\]
with the following product structure:
\[
\mathbf{y}^\mathbf{v}\cdot\mathbf{y}^{\mathbf{v}'}=t^{\chi(\mathbf{v}',\mathbf{v})}\mathbf{y}^{\mathbf{v}+\mathbf{v}'},\quad 
\mathbf{x}^\mathbf{w}\cdot\mathbf{x}^{\mathbf{w}'}=\mathbf{x}^{\mathbf{w}+\mathbf{w}'},\quad 
y_i\cdot x_j=t^{\delta_{i,j}}x_j\cdot y_i
\]
where $y_i:=\mathbf{y}^{[s_i]}$ and $x_j:=\mathbf{x}^{[P_i]}$.

We define the surjective algebra homomorphism $\pi\colon \DQT\twoheadrightarrow \mathrm{QT}^\vee$ by
\[
x_i \otimes 1  \longmapsto  x_i,\quad 1 \otimes y_i  \longmapsto  \mathbf{x}^{[s_i]}.
\]
The kernel of $\pi$ is generated by $\{(\mathbf{x}^{[s_i]}\otimes 1) - (1\otimes y_i)\mid i \in I\}$.

Let $\QT_{0}$, $\QT_{0}^\vee$ and $\DQT_{0}$ denote the $\C[t^\pm]$-subalgebra generated by $\mathbf{y}^\mathbf{v}$'s and $\mathbf{x}^\mathbf{w}$'s.
We put 
\[
\QT_{\mr{sc},\pm}:=\QT_0/(t\pm 1)\QT_0
\]
and define $\QT_{\mr{sc},\pm}^\vee$ and $\DQT_{\mr{sc},\pm}^\vee$ in the same way. 
Let $\{-,-\}$ denote the Poisson bracket on these quotients.

\subsubsection{Integration map}
\begin{thm}[\protect{\cite[Theorem 6.12]{joyce-2}, \cite[Theorem 6.3]{bridgeland-hall}}]\label{thm_integral}
There is a unique $L$-graded linear map 
\[
I_\pm\colon \mathrm{MH}_{\mr{sc}}(\A)\to \mathrm{QT}_{\mr{sc},\pm}
\]
such that if $X$ is a variety with a map $f\colon X\to\M$ factoring through $\M(\mathbf{v})\subset \M$ then
\begin{align*}
I_+\bigl([f\colon X\to\M]\bigr)&:=e(X)\cdot \mathbf{y}^{\mathbf{v}},\\
I_-\bigl([f\colon X\to\M]\bigr)&:=\biggl(\,\sum_{n\in\Z}n\cdot e\bigl((f^*\nu_{\mca{M}_\A})^{-1}(n)\bigr)\biggr)\cdot \mathbf{y}^{\mathbf{v}}
\end{align*}
Moreover, $I_\pm$ is a Poisson algebra homomorphism.
\end{thm}
\begin{conj}[\protect{\cite{ks}}]\label{conj_ks}
There exists\footnote{They construct $I_{\mr{KS}}$ using {\it motivic Milnor fiber} and give a proof of the conjecture modulo certain expected formula on motivic Milnor fibers.} an $L$-graded $\Lambda$-algebra homomorphism
\[
I_{\mr{KS}}\colon \mathrm{MH}(\A)\to \mathrm{QT}(\A)
\]
defined by taking ``motivic invariants''. 
\end{conj}

\begin{rem}
Since it is $L$-graded, the homomorphism $I$ (and $I_{\mr{KS}}$ if it exists) extends to the completion. We use the same symbol for the extended homomorphism.

\end{rem}

\subsection{Absence of poles}\label{subsec_absence}
Let $\mca{C}$ be one of the categories $\A$, $\A_\kk$, $\mca{T}_\kk$, $\mca{T}_\kk[-1]$, $\mca{T}_\kk^\bot$ and $\mca{S}(r)$.
As we showed in \S \ref{subsec_Tstr}, we have a Bridgeland's stability condition $(Z,\mca{P})$ on $\dfa$ such that
\[
\mca{P}((0,1])=\A,\quad \mca{P}((0,\phi])=\mca{C}
\]
for some $0<\phi\leq 1$.
By the results in \cite{joyce-1}, we get the algebraic moduli stacks $\mca{M}_\mca{C}$ of objects in $\mca{C}$.

We put
\begin{align}\label{eq_ep}
\ep_{\mca{C}}:=\log(1+\mca{M}_{\mca{C}}):=\sum_{l\geq 1}\frac{(-1)^l}{l}\mca{M}_{\mca{C}}*\cdots*\mca{M}_{\mca{C}}\in \hMH(\mca{C})
\end{align}
and $\tep_{\mca{C}}:=(\mathbb{L}-1)\ep_{\mca{C}}\in \hMH(\mca{C})$.
Then we have
\begin{equation}\label{eq_exp}
\mca{M}_{\mca{C}}=\exp(\ep_{\mca{C}}):=\sum_{l\geq 1}\frac{1}{l!}\,\ep_{\mca{C}}*\cdots*\ep_{\mca{C}}.
\end{equation}
\begin{thm}[\protect{\cite[Section 6.2]{joyce-4}}]\label{thm_pole}
$\tepk\in \hMH_{0}(\mca{C})$.
\end{thm}
We put
\[
\hep_\mca{C}:=\tep_\mca{C}|_{\mathbb{L}-1}\in \widehat{\mr{MH}}_{\mr{sc}}.
\]

\section{Proof}\label{sec_proof}
\subsection{Hall algebra identities}\label{subsec_eq_in_MH}
Throughout this subsection. let $\mca{C}$ be one of the categories $\A$, $\A_\kk$, $\mca{T}_\kk$, $\mca{T}_\kk[-1]$, $\mca{T}_\kk^\bot$ and $\mca{S}_{(r)}$.
Let $\mca{M}_\mca{C}$ be the moduli stack of objects in $\mca{C}$.

For an element $P\in\pera$ we define the following moduli stacks:
\[
\mathfrak{Hom}(P,{\mca{C}}) := \{(f,E)\mid E\in {\mca{C}}, f\in \Hom(P,E)\}.
\]
\begin{prop}[\protect{\cite[Lemma 4.3]{bridgeland-dtpt}}]\label{prop_mhsi}
\[
\mr{Hilb}_{J}(i)=\mf{Hom}(P_i,{\mca{A}})*\mca{M}_\A^{-1}\quad \text{\textup{(motivic Hilbert scheme identity)}}.
\]
\end{prop}
\begin{prop}[\protect{\cite[Lemma 4.1]{bridgeland-dtpt}}]\label{prop_mtpi}
\[
\mca{M}_{\A}=\mca{M}_{\mathcal{T}_{\kk}}*\mca{M}_{\mathcal{T}_{\kk}^\bot},\quad \mca{M}_{\A_\kk}=\mca{M}_{\mathcal{T}_{\kk}^\bot}*\mca{M}_{\mathcal{T}_{\kk}[-1]}\quad \text{\textup{(motivic torsion pair identity)}}.
\]
\end{prop}

For an element $R\in \A$, let $\mr{Grass}(R,{\Afd})$ be the moduli stack of elements in $\A$ equipped with surjections from $R$:
\[
\mr{Grass}(R,{\Afd})  := \{(f,E)\mid E\in {\Afd}, f\in \Hom(R,E), \text{$f$\,:\,surjective}\}.
\]
\begin{prop}\label{prop_grass}
\begin{equation}\label{eq_grass}
\mf{Hom}(R,{\Afd})=
\mf{Grass}(R,{\Afd})*
\mca{M}_{\Afd}^{-1}.
\end{equation}
\end{prop}
\begin{proof}
We can prove in the same way as the motivic Hilbert scheme identities.
\end{proof}
\begin{prop}\label{prop_hom2}
\begin{equation}\label{eq_hom2}
\mf{Hom}(P_{\kk,i}[1],{\mca{T}_\kk})*\mca{M}_{\mca{T}^\bot_\kk}
=
\mf{Hom}(P_{\kk,i}[1],\A).
\end{equation}
\end{prop}
\begin{proof}
As in \cite[\S 4]{bridgeland-dtpt}, a $\C$-valued point of $\mf{Hom}(P_{\kk,i}[1],{\mca{T}_\kk})*\mca{M}_{\mca{T}^\bot_\kk}
$ is represented by a diagram 
\[
\xymatrix{
        & P_{\kk,i}[1]\ar[d] &         &         & \\
0\ar[r] & Y\ar[r]            & X\ar[r] & Z\ar[r] & 0
}
\]
with $F\in \mca{T}_\kk$ and $Z \in \mca{T}{}^\bot_\kk$.
By composing the morphisms in the diagram, we get a family of morphism $P_{\kk,i}[1]\to X$ on $\mf{Hom}(P_{\kk,i}[1],{\mca{T}_\kk})*\mca{M}_{\mca{T}^\bot_\kk}
$, which induces a morphism from $\mf{Hom}(P_{\kk,i}[1],{\mca{T}_\kk})*\mca{M}_{\mca{T}^\bot_\kk}
$ to the moduli stack $\mf{Hom}(P_{\kk,i}[1],\A)$.

Since $\Hom(P_{\kk,i}[1],Z)=\Hom(P_{\kk,i}[1],Z[-1])=0$, we have 
\[
\Hom(P_{\kk,i}[1],X)=\Hom(P_{\kk,i}[1],Y).
\]
The axiom of the torsion pair and the equation above provide an equivalence of $\C$-valued points induced by the morphism from $\mf{Hom}(P_{\kk,i}[1],{\mca{T}_\kk})*\mca{M}_{\mca{T}^\bot_\kk}
$ to  $\mf{Hom}(P_{\kk,i}[1],\A)$.
\end{proof}
The following lemma is clear:
\begin{lem}
\begin{equation}\label{eq_hom3}
\mf{Hom}(P_{\kk,i}[1],\A)=\mf{Hom}(R_{\kk,i},\A).
\end{equation}
\end{lem}
The following equation plays a principal role in this paper:
\begin{prop}\label{prop_hom4}
We have 
\begin{equation}\label{eq_hom4}
\mf{Hom}(P_{\kk,i}[1],{\mca{T}_\kk})*\mca{M}_{\mca{T}_\kk}^{-1}=
\mr{Grass}(R_{\kk,i},\A)
\end{equation}
\textup{(motivic quiver Grassmannian identity)}. 
In particular, $\mf{Hom}(P_{\kk,i}[1],{\mca{T}_\kk})*\mca{M}_{\mca{T}_\kk}^{-1}\in \MH_0$.
\end{prop}
\begin{proof}
\begin{align*}
\mf{Hom}(P_{\kk,i}[1],{\mca{T}_\kk})*\mca{M}_{\mca{T}_\kk}^{-1}
\ \ =\  &\ 
\mf{Hom}(P_{\kk,i}[1],{\mca{T}_\kk})*
\mca{M}_{\mca{T}_\kk{}^\bot}*
\mca{M}_{\mca{T}_\kk^\bot}^{-1}*
\mca{M}_{\mca{T}_\kk}^{-1}\\
\overset{\eqref{eq_hom2}}{=} &\ 
\mf{Hom}(P_{\kk,i}[1],{\A})*\mca{M}_{\A}^{-1}\\
\overset{\eqref{eq_hom3}}{=} &\ 
\mf{Hom}(R_{\kk,i},{\A})*\mca{M}_{\A}^{-1}\\
\overset{\eqref{eq_grass}}{=} &\ 
\mr{Grass}(R_{\kk,i},{\A}[-1]).
\end{align*}
\end{proof}
\begin{prop}\label{prop_mfi}
We have
\[
\mca{M}_{\mathcal{T}_{\kk}}=\bigl(\mca{M}_{\mathcal{S}{(1)}}\bigr)^{\ep(1)}*\cdots *\bigl(\mca{M}_{\mathcal{S}{(l)}}\bigr)^{\ep(l)}\quad (\mr{motivic\ factorization\ identity}).
\]
\end{prop}
\begin{proof}
We can prove in the same way as the torsion pair identities.
\end{proof}
For $\mathbf{w}\in M$, we define
\[
\mca{M}_{\mca{C}}[\mathbf{w}]:= \sum_{\mathbf{v}}\mathbb{L}^{\chi(\mathbf{w},\mathbf{v})}\cdot\mca{M}_{\mca{C}}(\mathbf{v}).
\]
We put $\mathbf{w}_{\kk,i}:=[\Gamma_{\kk,i}]$.
\begin{prop}
\[
\mf{Hom}(P_{\kk,i},{\mca{C}})
=
\mca{M}_{\mca{C}}[\mathbf{w_{\kk,i}}]\]
\end{prop}
\begin{proof}

We can realize $\mca{M}_{\mca{C}}(\vv)$ as a quotient stack $[\X/GL(\vv)]$, where $GL(\vv)$ is a direct product of $\mr{GL}(v_i)$'s.
Note that $\mf{Hom}(\Gamma_i,{\mca{C}})(\vv)$ is a vector bundle of rank $v_i$ on $\mca{M}_{\mca{C}}(\vv)$, whose pull-back on $\X$ is trivial. 
Since $GL(\vv)$ is special, $\mf{Hom}(\Gamma_i,{\mca{C}})(\vv)$ is Zariski locally trivial.
\end{proof}
\begin{cor}\label{cor_hom_factor}
\[
\mf{Hom}(P_{\kk,i}[1],{\mca{T}_{\kk}})=\bigl(\mf{Hom}(P_{\kk,i}[1],{\mca{S}(1)})\bigr)^{\ep(1)}*\cdots *\bigl(\mf{Hom}(P_{\kk,i}[1],{\mca{S}(l)})\bigr)^{\ep(l)}.
\]
\end{cor}

\subsection{Idea}\label{subsec_idea}
The purpose of this subsection is to show the idea of the proof. 
In this subsection we {\it assume} Conjecture \ref{conj_ks} is true, since it would make the argument clearer.
The actual proof starts from the next subsection, which is independent from Conjecture \ref{conj_ks}.

We define the torus automorphism
\[
\begin{array}{ccccc}
\widehat{\mr{q\text{-}Ad}}_{\A}:=\mathrm{Ad}_{I_{\mr{KS}}(\mca{M}_{\mca{A}})} & \colon & \hDQT(\mca{A}) & \overset{\sim}{\longrightarrow} & \hDQT(\mca{A})\\
&& \bullet & \longmapsto & I_{\mr{KS}}(\mca{M}_\mca{A})\times \bullet \times I_{\mr{KS}}(\mca{M}_\mca{A})^{-1}.
\end{array}
\]

\begin{PROP}
\begin{itemize}
\item[\textup{(1)}]We have
\[
\widehat{\mr{q\text{-}Ad}}_{\mca{A}}(x_i)=
x_i\cdot I_{\mr{KS}}\Bigl(\Hom(P_i,\mca{A})*\mca{M}_{\mca{A}}^{-1}\Bigr).
\]
\item[\textup{(2)}]
We have
\[
\widehat{\mr{q\text{-}Ad}}_{\A}(x_i)=
x_i\cdot I_{\mr{KS}}(\mr{Hilb}_J(i)).
\]
In particular, the non-commutative Donaldson-Thomas invariants for $J$ are encoded in the torus automorphism $\widehat{\mr{q\text{-}Ad}}_{\A}$
\end{itemize}
\end{PROP}
\begin{proof}
Note that we have 
\begin{equation}\label{eq_commutator}
\mca{E}\cdot x_i=x_i\cdot \mca{E}|_{y_j=y_j\cdot t^{2\delta_{ij}}}
\end{equation}
for $\mca{E}\in\hQT(\A)$, where $\mca{E}|_{y_j=y_j\cdot t^{2\delta_{ij}}}$ is given by substituting $y_j\cdot t^{2\delta_{ij}}$ for $y_j$.
We call this as {\it the commutator identity}.
The first claim is a consequence of the commutator identity.
The second one follows from the ``motivic Hilbert scheme identity'' (Proposition \ref{prop_mhsi}).
\end{proof}
We define the torus automorphism
\[
\widehat{\mr{q\text{-}Ad}}_{\mca{C}}:=\mathrm{Ad}_{I_{\mr{KS}}(\mca{M}_{\mca{C}})}\colon\hDQT(\mca{C})\overset{\sim}{\longrightarrow}\hDQT(\mca{C})
\]
in the same way. 
\begin{PROP}
We have the factorization identities
\begin{align}
&\widehat{\mr{q\text{-}Ad}}_{\A}=
\widehat{\mr{q\text{-}Ad}}_{\mca{T}_\kk}\circ
\widehat{\mr{q\text{-}Ad}}_{\mca{T}_\kk^\bot},\label{eq_q_factor1}\\
&\widehat{\mr{q\text{-}Ad}}_{\A_\kk}=
\widehat{\mr{q\text{-}Ad}}_{\mca{T}_\kk{}^\bot}\circ
\widehat{\mr{q\text{-}Ad}}_{\mca{T}_\kk[-1]},\label{eq_q_factor2}.
\end{align}
In particular, $\widehat{\mr{q\text{-}Ad}}_{\mca{T}_\kk}$ and $\widehat{\mr{q\text{-}Ad}}_{\mca{T}_\kk[-1]}$ provide a transformation formula of non-commutative Donaldson-Thomas invariants between $J$ and $J_\kk$.
\end{PROP}
\begin{proof}
They are consequences of the ``motivic torsion pair identity'' (Proposition \ref{prop_mtpi}).
\end{proof}
\begin{PROP}
\begin{itemize}
\item[\textup{(1)}]
We have
\[
\widehat{\mr{q\text{-}Ad}}_{\mca{T}_\kk[-1]}(x_{\kk,i})=
x_{\kk,i}\cdot I_{\mr{KS}}\Bigl(\Hom(P_{\kk,i},\mca{T}_\kk[-1])*\mca{M}_{\mca{T}_\kk[-1]}^{-1}\Bigr).
\]
\item[\textup{(2)}]
The torus automorphism can be described in terms of quiver Grassmannians:
\[
\widehat{\mr{q\text{-}Ad}}_{\mca{T}_\kk[-1]}(x_{\kk,i})=
x_{\kk,i}\cdot I_{\mr{KS}}\bigl(\mr{Grass}(R_{\kk,i}[-1],\A[-1])\bigr).
\]
In particular, the transformation formula of the non-commutative Donaldson-Thomas invariants can be described in terms of quiver Grassmannians.
\end{itemize}
\end{PROP}
\begin{proof}
The first claim follows by the commutator identity. 
The second one follows by the ``motivic quiver Grassmannian identity'' (Proposition \eqref{eq_hom4}).
\end{proof}
\begin{PROP}
We have the factorization identity
\[
\widehat{\mr{q\text{-}Ad}}_{\mca{T}_\kk[-1]}=
\Bigl(\widehat{\mr{q\text{-}Ad}}_{\mca{S}(1)[-1]}\Bigr)^{\ep(1)}\circ\cdots\circ
\Bigr(\widehat{\mr{q\text{-}Ad}}_{\mca{S}(l)[-1]}\Bigr)^{\ep(l)}.
\]
\end{PROP}
\begin{proof}
This is a consequences of the ``motivic factorization identity'' (Proposition \ref{prop_mfi}).
\end{proof}

\subsection{Proof}\label{subsec_auto}
\subsubsection{Definition of the automorphism}\label{subsec_531}
Let $\mca{C}$ be one of the categories $\A$, $\A_\kk$, $\mca{T}_\kk$, $\mca{T}_\kk[-1]$, $\mca{T}_\kk^\bot$ and $\mca{S}(r)$.
Recall that we have $\hep_\mca{C}\in \widehat{\mr{MH}}_0(\mca{C})$.
\begin{defn}
\[
\widehat{\mr{Ad}}_{\mca{C},\sigma}
:=
\exp\Bigl(\mathrm{ad}_{I_\sigma\bigl(\hep_\mca{C}\bigr)}\Bigr)\colon \hDQT_{\mr{sc},\sigma}(\mca{C})\overset{\sim}{\too}\hDQT_{\mr{sc},\sigma}(\mca{C})
\]
\end{defn}
We will prove Theorem \ref{thm_520}, \ref{thm_521}, \ref{thm_factor1} and \ref{thm_factor2} which induce all the results in \S \ref{sec_statement}.

\subsubsection{Infinitesimal commutator identity}
For $\ep=\sum\ep(\mathbf{v})\in \hMH(\A)$ and $\hep=\sum\hep(\mathbf{v})\in \hMH_{\mr{sc}}(\A)$ we define $\ep[\mathbf{w}_{i}]\in \hMH(\A)$ and $\hep\{\mathbf{w}_i\}\in \hMH_{\mr{sc}}(\A)$ by
\[
\ep[\mathbf{w}_i]:=\sum \mathbb{L}^{\chi([\mathbf{w}_i],\mathbf{v})}\cdot \ep(\mathbf{v}),\quad
\hep\{\mathbf{w}_i\}:=\sum {\chi(\mathbf{w}_i,\mathbf{v})}\cdot \hep(\mathbf{v})
\]
respectively.
Then we have the following:
\begin{lem}
We have $\ep_{\A}[\mathbf{w}_i]-\ep_{\A}\in \hMH_{0}({\A})$ and 
\[
(\ep_{\A}[\mathbf{w}_i]-\ep_{\A})|_{\mathbb{L}=1}=\hep_{\A}\{\mathbf{w}_i\}.
\]
(See \S \ref{subsec_absence} for the definitions.)
\end{lem}
We put
\[
\mca{E}_{i,\A}^{\langle p\rangle}:=\sum_j\frac{(-1)^jp!}{j!(p-j)!}\ep_{\A}[\mathbf{w}_i]^{*(p-j)}*\ep_{\A}^{*(j)}\in \hMH(\A).
\]
\begin{lem}
$\mca{E}_{i,\A}^{\langle p\rangle}\in \widehat{\mr{MH}}_{0}(\A)$.
\end{lem}
\begin{proof}
Since we have
\[
\mca{E}_{i,\A}^{\langle p+1\rangle}:=
(\ep_\A[\mathbf{w}_i]-\ep_\A)\cdot \mca{E}_{i,\A}^{\langle p\rangle} + \frac{1}{\mathbb{L}-1}\Bigl[\tep_\A,\mca{E}_{i,\A}^{\langle p\rangle}\Bigr],
\]
the claim follows by induction.
\end{proof}
\begin{cor}
\begin{equation}\label{eq_517}
\mca{E}_{i,\A}^{\langle p+1\rangle}\big|_{\mathbb{L}=1}=
\hep_\A\{\mathbf{w}_i\}\cdot \mca{E}_{i,\A}^{\langle p\rangle}\big|_{\mathbb{L}=1} + \Bigl\{\hep_\A,\mca{E}_{i,\A}^{\langle p\rangle}\big|_{\mathbb{L}=1}\Bigr\}.
\end{equation}
\end{cor}
\begin{prop}\label{prop_commutator}
\begin{align*}
\widehat{\mr{Ad}}_{\A,\sigma}(x_{i,\sigma})&=x_{i,\sigma}\cdot I_\sigma\Bigl(\bigl(\mca{M}_{\A}[\mathbf{w}_i]*\mca{M}_{\A}^{-1}\bigr)\big|_{\mathbb{L}=1}\Bigr),\\
\widehat{\mr{Ad}}_{\A,\sigma}(y_{i,\sigma})&=y_{i,\sigma}\cdot \prod_{j}I_\sigma\Bigl(\bigl(\mca{M}_{\A}[\mathbf{w}_j]*\mca{M}_{\A}^{-1}\bigr)\big|_{\mathbb{L}=1}\Bigr)^{\bar{Q}(j,i)}.
\end{align*}
\end{prop}
\begin{proof}
We define ${E}_{i,\A}^{\langle p\rangle}\in \hQT_{\mr{sc},\sigma}(\A)$ by 
\[
\bigl(\{I_\sigma(\hep_\A),-\}\bigr)^p(x_{i,\sigma})={E}_{i,\A}^{\langle p\rangle}\cdot x_{i,\sigma}.
\]
By \eqref{eq_exp}, it is suffice to show that ${E}_{i,\A}^{\langle p\rangle}=I_\sigma\Bigl(\mca{E}_{i,\A}^{\langle p\rangle}\big|_{\mathbb{L}=1}\Bigr)$. Since we have
\begin{align*}
&\Bigl\{I_\sigma(\hep_\A),I_\sigma\bigl(\mca{E}_{i,\A}^{\langle p\rangle}\big|_{\mathbb{L}=1}\bigr)\cdot x_{i,\sigma}\Bigr\}\\
&\ =\ 
\Bigl(I_\sigma(\hep_\A\{\mathbf{w}_i\})\cdot I_\sigma\bigl(\mca{E}_{i,\A}^{\langle p\rangle}\big|_{\mathbb{L}=1}\bigr)
+\Bigl\{I_\sigma(\hep_\A), I_\sigma\bigl(\mca{E}_{i,\A}^{\langle p\rangle}\big|_{\mathbb{L}=1}\bigr) \Bigr\}\Bigr)\cdot x_{i,\sigma}\\
&\ =\ I_\sigma\biggl(\hep_\A\{\mathbf{w}_i\}\cdot \mca{E}_{i,\A}^{\langle p\rangle}\big|_{\mathbb{L}=1} + \Bigl\{\hep_\A, \mca{E}_{i,\A}^{\langle p\rangle}\big|_{\mathbb{L}=1} \Bigr\}\biggr)\cdot x_{i,\sigma}\\
&\overset{\eqref{eq_517}}{=}I_\sigma\bigl(\mca{E}_{i,\A}^{\langle p\rangle}\big|_{\mathbb{L}=1}\bigr)\cdot x_{i,\sigma},
\end{align*}
the first equation follows by induction.
The second one follows since we have
\[
\mca{M}_{\A}[\mathbf{w}+\mathbf{w}']*\mca{M}_{\A}^{-1}
=
\bigl(\mca{M}_{\A}[\mathbf{w}]*\mca{M}_{\A}^{-1}\bigr)[\mathbf{w}']*
\bigl(\mca{M}_{\A}[\mathbf{w}']*\mca{M}_{\A}^{-1}\bigr).
\]
\end{proof}
Similarly we have the following:
\begin{prop}\label{prop_commutator2}
Let $\mca{C}$ be one of $\A_\kk$, $\mca{T}_\kk$ and $\mca{S}{(r)}$. Then we have
\begin{align*}
\widehat{\mr{Ad}}_{\mca{C},\sigma}(x_{\kk,i,\sigma})&=x_{\kk,i,\sigma}\cdot I_\sigma\Bigl(\bigl(\mca{M}_{\mca{C}}[\mathbf{w}_{\kk,i}]*\mca{M}_{\mca{C}}^{-1}\bigr)\big|_{\mathbb{L}=1}\Bigr).\\
\widehat{\mr{Ad}}_{\mca{C},\sigma}(y_{\kk,i,\sigma})&=y_{\kk,i,\sigma}\cdot \prod_j I_\sigma\Bigl(\bigl(\mca{M}_{\mca{C}}[\mathbf{w}_{\kk,i}]*\mca{M}_{\mca{C}}^{-1}\bigr)\big|_{\mathbb{L}=1}\Bigr)^{\bar{Q}_\kk(j,i)}.
\end{align*}
\end{prop}

\subsubsection{Hilbert/Grassmann in the automorphisms}
\begin{thm}\label{thm_520}
\begin{align*}
\widehat{\mr{Ad}}_{\A,\sigma}(x_{i,\sigma})&=
x_{i,\sigma}\cdot \Biggl(\sum_\mathbf{v}e_\sigma\Bigr(\mathrm{Hilb_{J}}(i,\mathbf{v})\Bigr)\cdot\mathbf{y}_\sigma^{\mathbf{v}}\Biggr),\\
\widehat{\mr{Ad}}_{\A,\sigma}(y_{i,\sigma})&=
y_{i,\sigma}\cdot \prod_j\Biggl(\sum_\mathbf{v}e_\sigma\Bigr(\mathrm{Hilb_{J}}(j,\mathbf{v})\Bigr)\cdot\mathbf{y}_\sigma^{\mathbf{v}}\Biggr)^{\bar{Q}(j,i)}.
\end{align*}
\end{thm}
\begin{proof}
This is a consequence of the motivic Hilbert scheme identity (Proposition \ref{prop_mhsi}) and Proposition \ref{prop_commutator}.
\end{proof}
\begin{thm}\label{thm_521}
\begin{align*}
\widehat{\mr{Ad}}_{\mca{T}_\kk[-1],\sigma}(x_{\kk,i,\sigma})&=
x_{\kk,i,\sigma}\cdot \Biggl(\sum_\mathbf{v}e_\sigma\Bigr(\mathrm{Grass}(\kk,i,\mathbf{v})\Bigr)\cdot\mathbf{y}_\sigma^{-\mathbf{v}}\Biggr),\\
\widehat{\mr{Ad}}_{\mca{T}_\kk[-1],\sigma}(y_{\kk,i,\sigma})&=
y_{\kk,i,\sigma}\cdot \prod_j\Biggl(\sum_\mathbf{v}e_\sigma\Bigr(\mathrm{Grass}(\kk,j,\mathbf{v})\Bigr)\cdot\mathbf{y}_\sigma^{-\mathbf{v}}\Biggr)^{\bar{Q}(j,i)}.
\end{align*}
\end{thm}
\begin{proof}
This is a consequence of the motivic quiver Grassmannian identity (Proposition \ref{prop_hom4}) and Proposition \ref{prop_commutator2}.
\end{proof}
\begin{cor}
The automorphism $\widehat{\mr{Ad}}_{\mca{T}_\kk[-1],\sigma}$ preserves $\DQT_{\mr{sc},\sigma}$ and induces an automorphism of $\QT_{\mr{sc},\sigma}$.
\end{cor}
\begin{proof}
The first half is clear from Theorem \ref{thm_521} and the second half follows since $\mathrm{ad}$ preserves the kernel of the map given in \eqref{eq_pi}.
\end{proof}
\begin{defn}
Let ${\mr{Ad}}_{\mca{T}_\kk[-1],\sigma}$ denote the automorphism on $\QT_{\mr{sc},\sigma}$ induced by $\widehat{\mr{Ad}}_{\mca{C}_\kk,\sigma}$.
\end{defn}
\begin{ex}
We put
\[
x_{(r),i,\sigma}:=\mathbf{x}_\sigma^{[\Gamma_{(r),i}]},\quad 
y_{(r),i,\sigma}:=\mathbf{y}_\sigma^{[\sigma_{(r),i}]}.
\]
Then we have
\[
\mr{Ad}_{\mca{S}(k),\sigma}(x_{(r),i,\sigma})=
\begin{cases}
x_{(r),i,\sigma} & i\neq k,\\
x_{(r),k,\sigma}(1+(y_{(r-1),k,\sigma})^{-1}) & i=k.
\end{cases}
\]
This gives the cluster transformation for the quiver $Q_{(r-1)}$ \textup{(}see Example \ref{ex_cluster}\textup{)}.
\end{ex}

\subsubsection{Factorization identity}
\begin{lem}\label{lem_factor}
For $X\in \hMH_{0}(\mca{C})$ we have
\[
\mr{Ad}_{\mca{C},\sigma}(I_\sigma(X|_{\mathbb{L}=1}))=I_\sigma\Bigl(\bigl(\mca{M}_{\mca{C}}*X*\mca{M}_{\mca{C}}^{-1}\bigr)|_{\mathbb{L}=1}\Bigr)
\]
\end{lem}
\begin{proof}
Note that we have
\begin{align*}
\mca{M}_{\mca{C}}*{X}*\mca{M}_{\mca{C}}^{-1}|_{\mathbb{L}=1} 
&=
\bigl(\exp([\ep_{\mca{C}},-])({X})\bigr)|_{\mathbb{L}=1}\\
&=
\Bigl(\exp\Bigl(\frac{1}{\mathbb{L}-1}[\tep_{\mca{C}},-]\Bigr)({X})\Bigr)|_{\mathbb{L}=1}\\
&=
\exp(\{\hep_{\mca{C}},-\})({X}|_{\mathbb{L}=1}).
\end{align*}
Then the claim follows since $I_\sigma$ respects the Poisson bracket.
\end{proof}
\begin{thm}\label{thm_factor1}
\begin{align*}
\widehat{\mr{Ad}}_{\A,\sigma}=
\widehat{\mr{Ad}}_{\mca{T}_\kk,\sigma}\circ
\widehat{\mr{Ad}}_{\mca{T}_\kk^\bot,\sigma},\quad
\widehat{\mr{Ad}}_{\A_\kk,\sigma}=
\widehat{\mr{Ad}}_{\mca{T}_\kk^\bot,\sigma}\circ
\widehat{\mr{Ad}}_{\mca{T}_\kk[-1],\sigma}.
\end{align*}
\end{thm}
\begin{thm}\label{thm_factor2}
\[
{\mr{Ad}}_{\mca{T}_\kk[-1],\sigma}=
\Bigl({\mr{Ad}}_{\mca{S}(1)[-1],\sigma}\Bigr)^{\ep(1)}
\circ\cdots\circ
\Bigl({\mr{Ad}}_{\mca{S}(l)[-1],\sigma}\Bigr)^{\ep(l)}.
\]
\end{thm}
We will show the proof of 
\begin{equation}\label{eq_factorization}
{\mr{Ad}}_{\mca{T}_\kk,\sigma}=
\Bigl({\mr{Ad}}_{\mca{S}(1),\sigma}\Bigr)^{\ep(1)}
\circ\cdots\circ
\Bigl({\mr{Ad}}_{\mca{S}(l),\sigma}\Bigr)^{\ep(l)}
\end{equation}
which is equivalent to Theorem \ref{thm_factor2} (we can prove Theorem \ref{thm_factor1} in the same way). 
We put $\delta(r):=\mca{M}_{\mca{C}(r)}$.
First, we can see the following identity by induction with respect to $r$ using Proposition \ref{prop_commutator} and Lemma \ref{lem_factor} :
\begin{align*}
&\Bigl(\widehat{\mr{Ad}}_{\mca{C}(r),\sigma}\Bigr)^{\ep(r)}
\circ\cdots\circ
\Bigl(\widehat{\mr{Ad}}_{\mca{C}(l),\sigma}\Bigr)^{\ep(l)}(\mathbf{x}_\sigma^\mathbf{w})=\\
&\mathbf{x}_\sigma^\mathbf{w}\cdot \prod_{r'=r}^{l} 
I_\sigma\Bigl(\delta_{(r)}^{\ep(r)}*\cdots *\delta_{(r'-1)}^{\ep(r'-1)}*(\delta_{(r')}[\mathbf{w}])^{\ep(r')}*\delta_{(r')}^{-\ep(r')}*\cdots *\delta_{(r)}^{-\ep(r)}\Big|_{\mathbb{L}=1}\Bigr)
\end{align*}
Then we have
\begin{align*}
&\Bigl(\widehat{\mr{Ad}}_{\mca{C}(1),\sigma}\Bigr)^{\ep(1)}
\circ\cdots\circ
\Bigl(\widehat{\mr{Ad}}_{\mca{C}(l),\sigma}\Bigr)^{\ep(l)}(\mathbf{x}_\sigma^\mathbf{w})\\
& = \mathbf{x}_\sigma^\mathbf{w}\cdot \prod_{r'=1}^{l} 
I_\sigma\Bigl(\delta_{(1)}^{\ep(r)}*\cdots *\delta_{(r'-1)}^{\ep(r'-1)}*(\delta_{(r')}[\mathbf{w}])^{\ep(r')}*\delta_{(r')}^{-\ep(r')}*\cdots *\delta_{(1)}^{-\ep(1)}\Big|_{\mathbb{L}=1}\Bigr)\\
& = \mathbf{x}_\sigma^\mathbf{w}\cdot I_\sigma\Bigl((\delta_{(1)}[\mathbf{w}])^{\ep(1)}*\cdots *(\delta_{(l)}[\mathbf{w}])^{\ep(l)}*\delta_{(l)}^{-\ep(l)}*\cdots *\delta_{(1)}^{-\ep(1)}\Big|_{\mathbb{L}=1}\Bigr)\\
&= \mathbf{x}_\sigma^\mathbf{w}\cdot I_\sigma\Bigl(\mca{M}_{\mca{T}_\kk}[\mathbf{w}_{\kk,i}]*\mca{M}_{\mca{T}_\kk}^{-1}\big|_{\mathbb{L}=1}\Bigr)\\
&= {\mr{Ad}}_{\mca{T}_\kk,\sigma} (\mathbf{x}_\sigma^\mathbf{w}).
\end{align*}
Here we use  Proposition \ref{prop_mfi} and Corollary \ref{cor_hom_factor} for the second equation and Proposition \ref{prop_commutator} for the last one.


\section{Applications for cluster algebras}\label{sec_F_and_g}

\subsection{Quiver with the principal framing}
Let $Q\pf$ be the following quiver:
\begin{description}
\item{vertices} : $I\sqcup I^*$ where $I^*=\{1^*,\ldots,n^*\}$,
\item{arrows} : $\{\text{arrows in $Q$}\}\sqcup \{i^*\to i\mid i\in I\}$.
\end{description}
This is called the quiver with the principal framing associated to $Q$.

A potential $W$ of $Q$ can be taken as a potential of $Q\pf$. 
In the rest of this paper, we assume that $(Q\pf,W)$ is successively f-mutatable with respect to a sequence $\kk$.
We apply Theorem \ref{thm_tilting} for $(Q\pf,W)$ and $\kk$ to get the sequence $\ep'(1),\ldots,\ep'(l)$.
We put
\[
\Phi\pf_{\kk}:=\Phi\pf_{k_l,\ep'(l)}\circ\cdots\circ\Phi\pf_{k_l,\ep'(1)}\colon \bar{\mca{D}}\pf\overset{\sim}{\longrightarrow} \bar{\mca{D}}\pf_\kk.
\]
Let $\phi\pf_{\kk}$ denote the homomorphism induced on the lattices $L_{Q\pf}$ or $M_{Q\pf}$.
Let $\ep(1),\ldots,\ep(l)$ be the sequence associated to the QP $(Q,W)$ and the sequence $\kk$. 
\begin{prop}\label{prop_pf}
\begin{itemize}
\item[\textup{(1)}] $\ep(r)=\ep'(r)$  for any $r$.
\item[\textup{(2)}] $\phi_\kk\pf([s_{i^*}])=[s_{\kk,i^*}]$ for any $i^*\in I^*$.
\item[\textup{(3)}] $\phi_\kk\pf(L_Q)\subset L_{Q_\kk}$ and $\phi_\kk\pf|_{L_Q}\colon L_Q\to L_Q$ coincides with $\phi_\kk$.
\item[\textup{(4)}] $\phi_\kk\pf(M_Q)\subset M_{Q_\kk}$ and $\phi_\kk\pf|_{M_Q}\colon M_Q\to M_Q$ coincides with $\phi_\kk$.
\end{itemize}
\end{prop}
\begin{proof}
We prove all the claims together by induction with respect to the length of the sequence $\kk$. 

Assume the claim holds for the sequence $\kk^-=(k_1,\ldots,k_{l-1})$.
Note that $\varepsilon (l)$ (resp. $\varepsilon' (l)$) is determined by the condition 
\[
\varepsilon(l)\times (\phi_{\kk^-})^{-1}([s_{{\kk^-},k_l}])\in C_{\mca{A}}\subset L_Q\quad 
\text{(resp. $\varepsilon'(l)\times (\phi_{\kk^-}\pf)^{-1}([s_{{\kk^-},k_l}])\in C_{\mca{A}\pf}\subset L_{Q\pf}$)}.
\]
By the induction assumption (3), we get $\varepsilon (l)=\varepsilon' (l)$.

We assume $\varepsilon (l+1)=+$ (for the case of $\varepsilon (l+1)=-$, we can see in the same way). Then we have
\begin{align*}
Q_{\kk'}(k_{l+1},i^*) & = 
\chi([s_{{\kk'},k_{l+1}}],[s_{{\kk'},i^*}])\\
& = \chi((\phi_\kk\pf)^{-1}[s_{{\kk'},k_{l+1}}],(\phi_{\kk'}\pf)^{-1}[s_{{\kk'},i^*}])\\
& = \chi((\phi_{\kk'})^{-1}[s_{{\kk'},k_{l+1}}],[s_{i^*}])\\
& \geq 0
\end{align*}
for any $i^*$. Then the claims follow \eqref{eq_yy'} and \eqref{eq_phi_for_s}.
\end{proof}
\begin{rem}\label{rem_c}
{For a sequence $\kk$ and a vertex $i\in I$, the vector $(Q_{\kk}(i,j^*))_{j^*\in I^+}$ is so called the {\it $c$-vector}. Now we see that the $c$-vector is given by $(\phi_{\kk'})^{-1}[s_{\kk,i}]$.}
\end{rem}

We have the following triangulated categories :
\begin{description}
\item{$\bar{\mca{D}}^{\mr{pf}}$} : $=\mca{D}\Gamma_{(Q^\mr{pf},W)}$,
\item{${\mca{D}}^{\mr{pf}}$} : $=\mca{D}^\mr{fd}\Gamma_{(Q^\mr{pf},W)}$,
\item{$\bar{\mca{D}}'$} : the full subcategory of $\bar{\mca{D}}^{\mr{pf}}$ consisting of objects whose cohomologies are supported on $I$,
\item{$\mca{D}'$} : $={\mca{D}}^{\mr{pf}}\cap \bar{\mca{D}}'$
\end{description}
The canonical t-structure of $\bar{\mca{D}}^{\mr{pf}}$ induces t-structures of ${\mca{D}}^{\mr{pf}}$, $\bar{\mca{D}}'$ and $\mca{D}'$.
Let $\bar{\mca{A}}^{\mr{pf}}$, ${\mca{A}}^{\mr{pf}}$, $\bar{\mca{A}}'$ and $\mca{A}'$ denote the cores of t-structures. 
\begin{lem}\label{lem_82}
\begin{itemize}
\item[\textup{(1)}] $\Phi_\kk\pf(s_{i^*})=s_{\kk,i^*}$ for any $i^*\in I^*$.
\item[\textup{(2)}] $\bar{\mca{T}}_\kk\pf\subset \mca{A}'$.
\end{itemize}
\end{lem}
\begin{proof}
The first claim follows Lemma \ref{prop_pf} (2). 
For (2), let $s^{\mr{pf},(r)}$ be the spherical object in $\mca{A}\pf$ defined in the same way as $s^{(r)}$ (see \S \ref{subsec_comp}).
By Lemma \ref{prop_pf} (3), we have $[s^{\mr{pf},(r)}]\in L_Q$. Hence we have $s^{\mr{pf},(r)}\in \mca{A}'$ and so $\bar{\mca{T}}_\kk\pf\subset \mca{A}'$. 
\end{proof}
We put 
\[
\bar{\mca{A}}_\kk\pf:=(\Phi\pf_{\kk})^{-1}(\mr{Mod}J_{\mu_\kk(Q\pf,W)})
\]
and 
\[
\bar{\mca{A}}'_\kk:= \bar{\mca{A}}_\kk\pf \cap \bar{\mca{D}}',\quad
{\mca{A}}'_\kk:= \bar{\mca{A}}_\kk\pf \cap {\mca{D}}'.
\]
Then $\bar{\mca{A}}'_\kk$ (resp. ${\mca{A}}'_\kk$) coincides with the full subcategory of $\bar{\mca{A}}_\kk\pf$ consisting of objects supported on $I$ (with finite dimensional cohomologies). 

We set 
$\bar{\mca{T}}'_\kk:=\bar{\mca{T}}\pf_\kk$ (resp. ${\mca{T}}'_\kk:=\bar{\mca{T}}\pf_\kk$) and 
\[
\bar{\mca{F}}'_\kk:=
\bar{\mca{F}}\pf_\kk\cap \bar{\mca{A}}',\quad 
\text{(resp. ${\mca{F}}'_\kk:=
\bar{\mca{F}}\pf_\kk\cap {\mca{A}}'$)}.
\]
Then, $(\bar{\mca{T}}'_\kk,\bar{\mca{F}}'_\kk)$ (resp. $({\mca{T}}'_\kk,{\mca{F}}'_\kk)$) gives a torsion pair of $\bar{\mca{A}}'$ (resp. ${\mca{A}}'$) and 
the tilted t-structure coincides with $\bar{\mca{A}}'_\kk$ (resp. ${\mca{A}}'_\kk$).

\subsection{CC formula for $(Q,W)$ and $(Q\pf,W)$}
We put
\[
R\pf_{\kk,i}:=H^{1}_{\bar{\A}\pf}((\Phi\pf_{\kk})^{-1}(P\pf_{\kk,i}))
\]
and
\[
\mathrm{Grass}\pf(\mathbf{k};i,\mathbf{v})
:=\{R\pf_{\kk,i}\twoheadrightarrow V\mid V\in \A',\ [V]=\mathbf{v}\}.
\]
Applying Theorem \ref{thm_CC} for $(Q\pf,W)$ we get
\[
\mathrm{FZ}\pf_{i,\mathbf{k}}(\underline{X})=
X_{\kk,i}\cdot \Biggl(\sum_\mathbf{v}e\Bigr(\mathrm{Grass}\pf(\mathbf{k};i,\mathbf{v})\Bigr)\cdot\mathbf{Y}^{-\mathbf{v}}\Biggr)
\]
where $Y_j=X_{j^*}^{-1}\cdot \prod_i(X_i)^{\bar{Q}(i,j)}$.
Then we have
\begin{equation}\label{eq_F}
F_{\kk,i}(\underline{y}):={FZ}\pf_{\kk,i}(\underline{X})|_{X_i=1,X_{i^*}=y_i}=\sum_\mathbf{v}e\Bigr(\mathrm{Grass}\pf(\mathbf{k};i,\mathbf{v})\Bigr)\cdot\mathbf{y}^{\mathbf{v}}.
\end{equation}

On the other hand, we put
\[
R'_{\kk,i}:=H^{1}_{\bar{\A}'}((\Phi\pf_{\kk})^{-1}(P_{\kk,i}))
\]
and
\[
\mathrm{Grass}'(\mathbf{k};i,\mathbf{v})
:=\{R'_{\kk,i}\twoheadrightarrow V\mid V\in \A',\ [V]=\mathbf{v}\}.
\]

We will apply the same arguments as in \S \ref{sec_proof} for $\mca{D}'\Gamma$.
Let $\mf{Hom}_{\mca{D}'\Gamma}(P_{\kk,i},\mca{T}_\kk)$ be the moduli stack which parametrizes homomorphisms in $\mca{D}'\Gamma$ from $P_{\kk,i}$ to elements in $\mca{T}\pf_\kk$. 
We can verify all the lemmas and propositions in \S \ref{subsec_idea} and \S \ref{subsec_auto} if we replace $\mathrm{Grass}(\mathbf{k};i,\mathbf{v})$ with $\mathrm{Grass}'(\mathbf{k};i,\mathbf{v})$.
As a consequence, we get the following modification of the Caldero-Chapoton type formula for $(Q,W)$ :
\begin{equation}\label{eq_mCC}
\mathrm{FZ}_{i,\mathbf{k}}(\underline{x})=
x_{\kk,i}\cdot \Biggl(\sum_\mathbf{v}e\Bigr(\mathrm{Grass}'(\mathbf{k};i,\mathbf{v})\Bigr)\cdot\mathbf{y}^{-\mathbf{v}}\Biggr).
\end{equation}
where $(\underline{y})^{-\mathbf{v}}=\prod_j (y_j)^{-v_j}$ and $y_j=\prod_i(x_i)^{\bar{Q}(i,j)}$.

\begin{prop}\label{prop_gr}
\[
R\pf_{\kk,i}=R'_{\kk,i}
\]
\end{prop}
\begin{proof}
Since we have no non-trivial morphism from 
\[
\mr{ker}\left( P\pf_{\kk,i}\twoheadrightarrow R\pf_{\kk,i}  \right)
\]
to $R'_{\kk,i}$, the composition 
\[
P\pf_{\kk,i}\twoheadrightarrow P_{\kk,i}\twoheadrightarrow R'_{\kk,i}
\]
factors through $P\pf_{\kk,i}\twoheadrightarrow R\pf_{\kk,i}$ :
\[
\xymatrix{
P\pf_{\kk,i}\ar@{->>}[r]\ar@{->>}[d] & R\pf_{\kk,i}\ar@{-->>}[d]\\
P_{\kk,i}\ar@{->>}[r] & R'_{\kk,i}.
}
\]

On the other hand, since $R\pf_{\kk,i}$ is supported on $I$ the surjection $P_{\kk,i}\pf\twoheadrightarrow R\pf_{\kk,i}$ factors through $P_{\kk,i}$. 
By the same reason, this map factors through $P_{\kk,i}\twoheadrightarrow R'_{\kk,i}$.
\[
\xymatrix{
P\pf_{\kk,i}\ar@{->>}[r]\ar@{->>}[d] & R\pf_{\kk,i}\\
P_{\kk,i}\ar@{->>}[r]\ar@{-->>}[ru] & R'_{\kk,i}\ar@{-->>}[u]
}.
\]
These two morphisms are the inverse of each other.  
\end{proof}

\subsection{$F$-polynomials and $g$-vectors}\label{subsec_F_and_g}
Combining \eqref{eq_F}, \eqref{eq_mCC} and Proposition \ref{prop_gr}, we get
\begin{equation}\label{eq_Fg}
\mathrm{FZ}_{i,\mathbf{k}}(\underline{x})=
x_{\kk,i}\cdot F_{\kk,i}(\underline{y}^{-1}).
\end{equation}
Finally, we have the following description of $F$-polynomials and $g$-vectors :
\begin{thm}\label{thm_Fg}
\begin{itemize}
\item[\textup{(1)}]
\[
F_{\kk,i}(\underline{y})=
\sum_\mathbf{v}e\Bigr(\mathrm{Grass}'(\mathbf{k};i,\mathbf{v})\Bigr)\cdot\mathbf{y}^{\mathbf{v}}
\]
\item[\textup{(2)}]
\[
g_{\kk,i} = \phi_{\kk}^{-1}([\Gamma_{\kk,i}])\in M_Q.\label{eq_g}
\]
\end{itemize}
\end{thm}
Since $L_{J}\otimes \R$ and $M_{J}\otimes \R$ are dual to each other via $\chi$ and $\phi_\kk$ preserves $\chi$, we get the following description of the $\teng$-vector. 
\begin{cor}\label{cor_619}
\[
\teng_{\kk,i} = \phi_{\kk}([s_i]) \in L_{Q_\kk}.
\]
\end{cor}
\begin{rem}
The $g$-vector can be viewed as a tropical counterpart of the $x$-variable, while the $c$-vector can be viewed as a tropical counterpart of the $y$-variable. 
The duality between the $g$- and the $c$-vectors is called toropical duality in \cite{nakanishi-zelevinsky} \footnote{The duality has proved in \cite{nakanishi-periodicity} for skewsymmetric matrices. For skewsymmetrizable matrices, it is still a conjecture.}. 
From our view point, the $x$-variable corresponds to the ``projective" $\Gamma_i$ and the $y$-variable corresponds to the simple $s_i$, and the toropical duality is a consequence of the duality between $\{\Gamma_i\}$ and $\{s_i\}$, 
\end{rem}

\subsubsection{Conjectures on $F$-polynomials}
The following claims follow directly from the description in Theorem \ref{thm_Fg}.
\begin{thm}[\protect{\cite[Conjecture 5.4]{fomin-zelevinsky4},\cite[Theorem 1.7]{DWZ2}}]
Each polynomial $F_{\kk,i}(\underline{y})$ has constant term $1$.
\end{thm}
\begin{thm}[\protect{\cite[Conjecture 5.5]{fomin-zelevinsky4},\cite[Theorem 1.7]{DWZ2}}]
Each polynomial $F_{\kk,i}(\underline{y})$ has a unique monomial of maximal degree. 
Furthermore, this monomial has coefficient $1$, and it is divisible by all the
other occurring monomials.
\end{thm}

\subsubsection{Conjectures on $g$-vectors}
\begin{thm}[\protect{\cite[Conjecture 7.10(2)]{fomin-zelevinsky4},\cite[Theorem 1.7]{DWZ2}}]
For any sequence $\kk$, the vectors $\{g_{\kk,i}\}_{i\in I}$ form a $\Z$-basis of the lattice $\Z^n$.
\end{thm}
\begin{proof}
This is clear from Theorem \ref{thm_Fg} (2).
\end{proof}
\begin{thm}[\protect{\cite[Conjecture 6.13]{fomin-zelevinsky4},\cite[Theorem 1.7]{DWZ2}}]
For any sequence $\kk$ and a vertex $i\in I$, the components of the vector $\teng_{\kk,i}$ are either all non-negative, or all non-positive.
\end{thm}
\begin{proof}
In the same way as Theorem \ref{thm_tilting}, we can see that $\Phi_\kk(s_i)\in \mca{A}_\kk$ or $\Phi_\kk(s_i)\in \mca{A}_\kk[1]$. 
Then, the claim is a consequence Corollary \ref{cor_619}.
\end{proof}
\begin{thm}[\protect{\cite[Conjecture 7.12]{fomin-zelevinsky4},\cite[Theorem 1.7]{DWZ2}}]
For a sequence $\kk=(k_0,\ldots,k_l)$, we take a new sequence 
\[
\kk^\circ:=(k_1,\ldots,k_l).
\]
Then we have
\[
\teng_{\kk,i^\circ}=\begin{cases}
-\teng_{i,{\kk}} & i=k_0, \\
\teng_{i,{\kk}} + Q(i,k_0)\cdot \teng_{k_0,{\kk}} & i\neq k_0,\ \ep(0)=-,\\
\teng_{i,{\kk}} + Q(k_0,i)\cdot \teng_{k_0,{\kk}} & i\neq k_0,\ \ep(0)=+.
\end{cases}
\]
\end{thm}
\begin{proof}
This is a consequence of \eqref{eq_phi_for_s} and Corollary \ref{cor_619}.
\end{proof}

\subsection{$g$-vectors determine $F$-polynomials}\label{subsec_g_to_F}
We define
\[
\zeta' \colon M_\mathbb{R} \to \mathrm{Stab}(\mca{D}'\Gamma).
\]
in the same way as \S \ref{subsec_embed}.
For $\theta\in M_\mathbb{R}$, let $\mca{A}'_\theta$ denote the
core of the t-structure corresponding to $\zeta'(\theta)$ and $(\mca{T}'_\theta,\mca{F}'_\theta)$ be the corresponding torsion pair.
For $\theta\in C_{\A_\kk}^*$, we have $\mathcal{A}_{\zeta(\theta)}=\mathcal{A}'_{\kk}$.

\begin{thm}[\protect{\cite[Conjecture 7.10(1)]{fomin-zelevinsky4},\cite[Theorem 1.7]{DWZ2}}]
Suppose we have
\[
\sum_{i\in J}a_i\cdot g_{\kk,i}=
\sum_{i\in J'}a'_i\cdot g_{\kk,i'}
\]
for some nonempty subsets $J,J'\subset I$ and some positive real numbers $a_i$ and $a'_i$. 
Then there is a bijection $\kappa\colon J\to J'$ such that for every $i\in J$ we have
\[
a_i=a'_{\kappa(i)},\quad 
g_{\kk,i}=g_{\kappa(i),\kk'},\quad 
F_{\kk,i}=F_{\kappa(i),\kk'}.
\]
In particular, $F_{\kk,i}$ is determined by $g_{\kk,i}$.
\end{thm}
\begin{proof}
By Theorem \ref{thm_Fg}, $g_{\kk,i}$ is primitive and 
\[
g_{\kk,i}\in \bigcap_{j\neq i}W_{\kk,j}\cap \overline{C_\kk^*}
\]
where 
\[
W_{\kk,j}:=\{\theta\in M_\mathbb{R}\mid \langle \theta,[s_{\kk,i}]\rangle=0\}.
\]
Then we have 
\[
\mr{Int}\left(\bigcap_{j\notin J}W_{\kk,j}\cap \overline{C_\kk^*}\right)
=
\left\{\sum_{i\in J}a_i\cdot g_{\kk,i}\mid a_i>0\right\}.
\]
The bijection $\kappa\colon J\to J'$ and
\[
a_i=a'_{\kappa(i)},\quad 
g_{\kk,i}=g_{\kk',\kappa(i)}
\]
follow from this description.

Let $\mca{C}_{\kk,J}$ be the full subcategory of $J_\kk$-modules supported on 
\[
\{i\mid i\neq J,\ep(\kk,i)=+\}.
\]
\begin{figure}[htbp]
  \centering
  \input{fig9.tpc}
  \caption{$\mca{A}'_\theta$}
  \label{}
\end{figure}
Then we have
\[
\A'_\theta=\A'_\kk(\mca{C}_{\kk,J}[1],{}^\bot\mca{C}_{\kk,J})\footnote{The category $\A\pf_\theta$ is not a module category in general.}.
\]
We define 
\[
\mathrm{Ad}\pf_{\mca{T}'_{\theta}}, \mathrm{Ad}\pf_{\mca{C}_{\kk,J}}\colon \QT\pf_{\mr{sc}}\overset{\sim}{\too} \QT\pf_{\mr{sc}}
\]
in the same way in \S \ref{subsec_531}, then we have
\[
\mathrm{Ad}\pf_{\mca{T}'_{\theta}}\circ\mathrm{Ad}\pf_{\mca{C}_{\kk,J}}=\mathrm{Ad}\pf_{\mca{T}'_{\kk}}
\]
as Theorem \ref{thm_factor2}.

Note that $\mathrm{Ad}\pf_{\mca{C}_{\theta}}$ depends on $\theta$, but not on $\kk$. 
Hence we get $F_{\kk,i}=F_{\kappa(i),\kk'}$.
\end{proof}


\providecommand{\bysame}{\leavevmode\hbox to3em{\hrulefill}\thinspace}
\providecommand{\MR}{\relax\ifhmode\unskip\space\fi MR }
\providecommand{\MRhref}[2]{%
  \href{http://www.ams.org/mathscinet-getitem?mr=#1}{#2}
}
\providecommand{\href}[2]{#2}

\end{document}